\documentclass[a4paper]{article}
\usepackage{amsthm,amssymb,amsmath,enumerate,graphicx,epsf,bbold,dsfont}

\newcommand{\COLORON}{0}
\newcommand{\NOTESON}{0}
\newcommand{\Debug}{0}
\usepackage[usenames,dvips]{color} 
\usepackage{amsthm,amssymb,amsmath,enumerate,graphicx,epsf}

\newcommand{\comment}[1]{}
\newcommand{\COMMENT}[1]{}

\definecolor{darkgray}{rgb}{0.3,0.3,0.3}
\newcommand{\defi}[1]{{\color{darkgray}\emph{#1}}}

\newcommand{\acknowledgement}{\section*{Acknowledgement}}

\comment{
	\begin{lemma}\label{}	
\end{lemma}
\begin{proof}

\end{proof}

\begin{theorem}\label{}
\end{theorem} 
\begin{proof} 	

\end{proof}

}

\newtheorem{proposition}{Proposition}[section]
\newtheorem{definition}[proposition]{Definition}
\newtheorem{theorem}[proposition]{Theorem}
\newtheorem{corollary}[proposition]{Corollary}
\newtheorem{lemma}[proposition]{Lemma}

\newtheorem{examp}[proposition]{Example}

\newcommand{\FIG}{0}

\ifnum \NOTESON = 1 \newcommand{\note}[1]{ 

	\ 

	{\color{blue} \hspace*{-60pt} NOTE: \color{Turquoise}{\small  \tt \begin{minipage}[c]{1.1\textwidth}  #1 \end{minipage} \ignorespacesafterend }} 
	
	\ 
	
	}
\else \newcommand{\note}[1]{} \fi

\newcommand{\afsubm}[1]{ \ifnum \Debug = 1 {\mymargin{#1}}
\fi} 

\ifnum \Debug = 1 \newcommand{\sss}{\ensuremath{\color{red} \bowtie \bowtie \bowtie\ }}
\else \newcommand{\sss}{} \fi

\ifnum \FIG = 1 \newcommand{\fig}[1]{Figure ``{#1}''}
\else \newcommand{\fig}[1]{Figure~\ref{#1}} \fi

\ifnum \Debug = 1 \usepackage[notref,notcite]{showkeys}
\fi

\ifnum \COLORON = 0 \renewcommand{\color}[1]{}
\fi

\newcommand{\showFig}[2]{
   \begin{figure}[htbp]
   \centering
   \noindent
   \epsfbox{#1.eps}
   \caption{\small #2}
   \label{#1}
   \end{figure}
}

\newcommand{\N}{\ensuremath{\mathds N}}
\newcommand{\R}{\ensuremath{\mathds R}}

\newcommand{\cc}{\ensuremath{\mathcal C}}

\newcommand{\cf}{\ensuremath{\mathcal F}}

\newcommand{\ch}{\ensuremath{\mathcal H}}

\newcommand{\ck}{\ensuremath{\mathcal K}}

\newcommand{\cm}{\ensuremath{\mathcal M}}

\newcommand{\cp}{\ensuremath{\mathcal P}}

\newcommand{\oo}{\ensuremath{\omega}}

\newcommand{\bet}{\ensuremath{\beta}}

\newcommand{\del}{\ensuremath{\delta}}
\newcommand{\eps}{\ensuremath{\epsilon}}

\newcommand{\fcg}{\ensuremath{|G|}}

\newcommand{\sm}{\backslash}

\newcommand{\sgl}[1]{\ensuremath{\{#1\}}}
\newcommand{\pth}[2]{\ensuremath{#1}\text{--}\ensuremath{#2}~path}

\newcommand{\arc}[2]{\ensuremath{#1}\text{--}\ensuremath{#2}~arc}

\newcommand{\seq}[1]{\ensuremath{(#1_n)_{n\in\N}}} 
 
\newcommand{\flo}[2]{\ensuremath{#1}\text{--}\ensuremath{#2}~flow} 

\newcommand{\g}{\ensuremath{G\ }}
\newcommand{\G}{\ensuremath{G}}

\newcommand{\ceil}[1]{\ensuremath{\lceil #1 \rceil}}

\newcommand{\ltp}{\ensuremath{|G|_\ell}}

\newcommand{\lER}{\ensuremath{\ell: E(G) \to \R_{>0}}}

\newcommand{\ksl}{Kirchhoff's cycle law}
\newcommand{\cutr}{non-elusive} 

\newcommand{\Lr}[1]{Lemma~\ref{#1}}
\newcommand{\Tr}[1]{Theorem~\ref{#1}}
\newcommand{\Sr}[1]{Section~\ref{#1}}
\newcommand{\Prr}[1]{Pro\-position~\ref{#1}}

\newcommand{\Cr}[1]{Corollary~\ref{#1}}

\newcommand{\Dr}[1]{De\-fi\-nition~\ref{#1}}

\newcommand{\lf}{locally finite}

\newcommand{\bos}{basic open set}

\renewcommand{\iff}{if and only if}
\newcommand{\fe}{for every}
\newcommand{\Fe}{For every}

\newcommand{\st}{such that}

\newcommand{\ti}{there is}

\newcommand{\obda}{without loss of generality}

\newcommand{\wrt}{with respect to}

\newcommand{\labtequ}[2]{ \begin{equation} \label{#1} 	\begin{minipage}[c]{0.9\textwidth}  #2 \end{minipage} \ignorespacesafterend \end{equation} }

\newcommand{\mymargin}[1]{
  \marginpar{%
    \begin{minipage}{\marginparwidth}\small%
      \begin{flushleft}%
        {\color{blue}#1}%
      \end{flushleft}%
   \end{minipage}%
  }%
}%

\newcommand{\mySection}[2]{}

\newcommand{\gn}{\ensuremath{G_n}}
\newcommand{\gseq}{graph approximation}

\newcommand{\pe}{pseudo-edge}
\newcommand{\gl}{graph-like}
\newcommand{\Gl}{Graph-like}

\newcommand{\des}{disconnecting edge-set}

\newcommand{\Hm}{\ensuremath{\ch}}
\newcommand{\HM}{Hausdorff measure}
\newcommand{\finhm}{\ensuremath{\Hm(X)< \infty}}

\newcommand{\cls}[1]{\overline{#1}}

\newcommand{\cov}{\tau}
\newcommand{\e}[2]{\mathds{E}_{#1}\left[#2\right]}
\newcommand{\en}[2]{\mathds{E}^n_{#1}\left[#2\right]}
\newcommand{\p}[2]{\mathds{P}_{#1}\left[#2\right]}

\renewcommand{\Pr}{\mathds{P}}
\newcommand{\8}{\infty}
\usepackage{dsfont}
\newcommand{\1}[1]{{\mathds1}_{\left[#1\right]}}

\newcommand{\BM}{Brownian motion}

\newcommand{\leth}{large enough that}
\newcommand{\seth}{small enough that}
\newcommand{\smp}{strong Markov property}

\newcommand{\Ex}{\ensuremath{\mathds{E}}}
\newcommand{\Exp}[1]{\ensuremath{\mathds{E}}[#1]}

\newcommand{\readi}[1]{\medskip {\it #1} \medskip}

\title{Brownian Motion on graph-like spaces}
\author{Agelos Georgakopoulos\thanks{Supported by FWF grant P-24028-N18 and EPSRC grant EP/L002787/1.}\\ 
  {Mathematics Institute}\\
 {University of Warwick}\\
  {CV4 7AL, UK}
\and Konrad Kolesko\thanks{Supported by NCN grant DEC-2012/05/B/ST1/00692}\\Instytut Matematyczny\\ Uniwersytet Wroc\l{}awski\\ pl. Grunwaldzki 2/4\\ 50-384 Wroc\l{}aw, Poland\\
}

\begin{document}
\maketitle

\begin{abstract}
We construct \BM\ on a wide class of metric spaces similar to graphs, and show that its cover time admits an upper bound depending only on the  length of the space.
\end{abstract}

\section{Introduction}

The aim of this paper is to construct the analog of \BM\ on metric spaces that are similar to graphs in a sense made precise below, and study some of its basic properties. It turns out that, under mild conditions, \ti\ a unique stochastic process qualifying for this. 

\fig{examples} shows some example spaces on which our process can live; the numbers indicate the lengths of the corresponding arcs. 

The first one is the Hawaian earring: an infinite sequence of circles attached to a common point $p$ to which they converge. 
It might at first sight seem impossible to have a \BM\ on this space started at $p$, unless we impose some ad-hoc bias as to the probability with which each circle is chosen first. However, there need not be a `first' circle visited by a continuous path from $p$, and indeed our process will traverse infinitely many of them before moving to any distance $r>0$ from $p$. Still, each of the finitely many points at distance exactly $r$ from $p$ has the same probability to be reached first. The second example is an \R-tree of finite total length. Our \BM\ will reach the `boundary' at the top after some finite time $\tau$, and will continue its continuous path after this, almost surely visiting infinitely many boundary points in any inteval $[\tau, \tau+\eps]$. The third example is obtained from the Sierpinski gasket by replacing articulation points with arcs. This space contains a homeomorphic copy of the second example, and a subspace homotopy equivalent to the first example; our process on it is more complex, combining features of both the above.
\epsfxsize=\hsize
\showFig{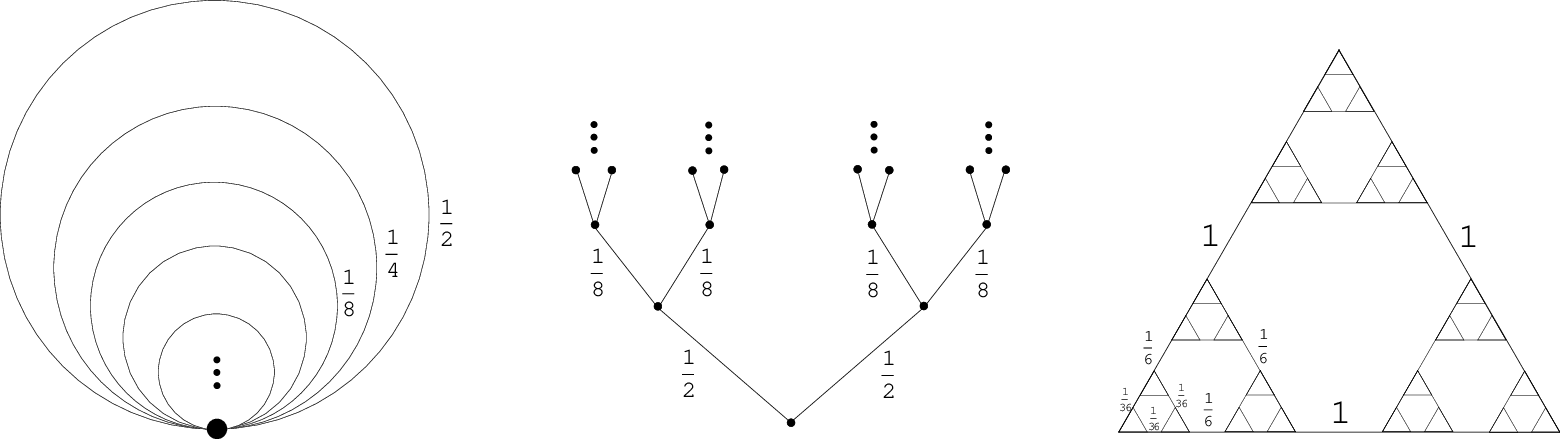}{Examples of \gl\ spaces.}

In all these examples, and in much greater generality indeed, our process behaves locally like standard \BM\ on a real interval $I$ on each open arc of our space isometric to $I$, its sample paths are continuous, it has the \smp, and it almost surely covers the whole space after finite time.

\medskip
We call a topological space $X$ \defi{\gl}, if it contains a set $E$ of pairwise disjoint copies of \R, called \defi{edges}, each of which is open in $X$, 
 \st\ the subspace $X\sm \bigcup E$ is  totally disconnected. This notion was introduced by Thomassen and Vella \cite{ThomassenVellaContinua}, and was motivated by recent developments in graph theory; see also \cite{gl}.

Recall that a \defi{continuum} is a compact, connected, non-empty metrizable space (some authors replace `metric' by Hausdorff). We will use $\Hm(X)$ to denote the 1-dimensional \HM\ of $X$. Although our processes can be constructed on any \gl\ continuum, for its uniqueness it is necessary to have $\Hm(X)< \infty$.

In order to construct our process, we use a result from \cite{gl} stating, roughly speaking, that every \gl\ space $X$ can be approximated by a sequence of finite graphs (i.e.\ 1-complexes) contained in $X$. Such a sequence of graphs is called a \gseq\ of $X$; see \Sr{secCon} for the precise definition. For example, any sequence \seq{G} where $G_n$ consists of finitely many of the cicles of the Hawaian earring and each circle appears in almost every $G_n$ is a \gseq. The main goal of this paper is to show that if $B_n$ denotes \BM\ on the $n$th member of any \gseq\ of $X$, then the $B_n$ converge weakly ---in the space of measures on continuous paths on $X$, see \Sr{measures}--- to a stochastic process $B$ on $X$ with all the desired properties, and this $B$ does not depend on the choice of the \gseq:

\begin{theorem} \label{thmain}
Let \g be a \gl\ continuum with $\Hm(X)< \infty$, and $o$ a point of $X$. 
Then \ti\ a stochastic process $B$ on $X$ with continuous sample paths starting at $o$, the \smp, and a stationary distribution  proportional to $\Hm$. 

Moreover, \fe\ \gseq\  \seq{G} of $X$, and every choice of points $o_n\in G_n$ \st\ $\lim o_n = o$, if $B_n$ is the standard \BM\ on $G_n$ from $o_n$, then $B_n$ converges weakly to $B$, and $B$ is unique with this property.

\end{theorem}

\Tr{thmain} states that our process is unique with the property of being a weak limit of \BM s on \gseq s of $X$, but we suspect that it is unique in a stronger sense\sss.

\medskip
It was shown in \cite{edgecov} that the expected time for \BM\ on a finite, connected 1-complex \g to cover all of \g is bounded from above by a value depending only on the total length of $G$ and not on its structure. Applying this to each member of our \gseq s, we prove the corresponding result for our \BM\ on an arbitrary \gl\ continuum: 

\begin{theorem} \label{ct}
The expected cover time of the process $B$ of \Tr{thmain} is at most $20 \Hm(X)^2$.
\end{theorem}
A related result of Krebs \cite{Krebs} shows that the hitting times for \BM\ on nested fractals are bounded.

\medskip

There are many  constructions of \BM\ on spaces similar to the ones considered in this paper: on finite graphs \cite{BaChaEqu}, on trees and their boundaries \cite{AlEvDir,BaxMar,BGPW,KigDir} on the Sierpinski gasket \cite{BaPeBro,goldstein_random_1987,KumBro} and many other fractals \cite{Hattori,hambly_brownian_1997,kusuoka_lecture_1993}. Brownian motion especially on fractals has attracted a lot of interest, with motivation coming both from pure mathematics and mathematical physics (see \cite{KumBro} and references therein), and has many  connections to other analytic properties of fractals which also attract a lot of research \cite{KigamiAnalysis,strichartz1999analysis}.

The first author had asked for a construction of \BM\ on a special type of \gl\ spaces, namely metric completions of infinite graphs \cite[Section 8]{AgCurrents}, and this paper gives a very satisfactory answer to that question.

\bigskip
This paper is structured as follows. After reviewing some definitions and basic facts \Sr{prel}, we prove the existence part of \Tr{thmain} in \Sr{secCon}. The uniqueness part is then proved in \Sr{SeqUniq}. Then we prove that our process has the \smp\ (\Tr{smp}), and the bound on the cover time is given in \Sr{cover}. Finally, we prove that \Hm\ is a stationary distribution and that our process behaves locally like standard \BM\ inside any edge in \Sr{further}.

\section{Preliminaries} \label{prel}

\subsection{\Gl\ spaces} \label{sgl}
 
An \defi{edge} of a topological space $X$ is an open subspace $I\subseteq X$ homeomorhpic to the real interval $(0,1)$ \st\ the closure of $I$ in $X$ is homeomorphic to $[0,1]$. (We could allow the closure of $I$ to be a circle; it is only for convenience in certain situations that we disallow this.)
Note that the frontier of an edge consists of two points, which we call its \defi{endvertices}. An edge-set of a topological space $X$ is a subspace consisting of finitely many, pairwise disjoint, edges of $X$.

A topological space $X$ is \defi{\gl} if \ti\ an edge-set $E$ of $X$ such that $G\sm E$ is totally disconnected. In that case, we call $E$ a \defi{\des}. 

The following fact provides an equivalent definition of a \gl\ con\-tinuum.

\begin{lemma}[\cite{gl}] \label{arc}
A continuum $X$ is \gl\ \iff\ \fe\ \eps\ \ti\ a finite set of edges $S_\eps$ of $X$ \st\ the diameter of every component of $X \sm  S_\eps$ is less than \eps.
\end{lemma}

The following property of \gl\ spaces is very useful to us, as it implies that \BM\ on such a space cannot travel a long distance without traversing a long edge.
\begin{proposition} \label{mineps}
If $X$ is a \gl\ continuum, then \fe\ $\rho>0$ \ti\ a finite edge-set $R_\rho$ of $X$ 
\st\ \fe\ topological path $p:[0,1]\to X$ in $X$, if $d(p(0),p(1))> \rho$ then $p$ traverses an edge in $R_\rho$.
\end{proposition}
\begin{proof}
Applying \Lr{arc} for $\eps=\rho/3$, we obtain a finite set of edges $S$ \st\ the diameter of every path-component of $X \sm  S$ is less than $\rho/3$. Subdivide each edge $e\in S$ into a finite set of edges each of length at most $\rho/6$, and let $R$ be the set of edges resulting from $S$ after all these subdivisions. Now note that any  topological path $p$  as in the assertion has to traverse an element of $R$; to see this, contract each path-component of $X \sm  S$ into a point to obtain a new metric space $X'$, and note that $X'$ is isometric to a finite graph whose edgeset can be identified with $R$. Moreover, after the  contractions we have $d(p(0),p(1))>\rho - 2\rho/3= \rho/3$, and as each  edge of our graph has length at least $\rho/6$, the assertion easily follows by geometric arguments. Thus we can choose  $R_\rho=R$.
\end{proof}

\Gl\ spaces have nice bases:
\begin{lemma}[\cite{gl}] \label{basis}
Let $X$ be a \gl\ metric continuum. Then the topology of $X$ has a basis consisting of connected open sets $O$ \st\ the frontier of $O$ is a finite set of points each contained in an edge.
\end{lemma} 

\subsection{Measures on the space of sample paths and weak convergence} \label{measures}
Given a \gl\ space $(X,d_X)$, we denote by $C=C_T(X)$ the set of continuous functions from the real interval $[0,T]$ to $X$. We call $C$ the \defi{space of sample paths}; our process will be formally defined as a probability measure on $C$. We endow $C$ with the $L^\infty$ metric 
$d_C(b,d) := \sup_{t\in [0,T]} d_X(b(t),c(t))$.

Let $\cm= \cm(C)$ denote the space of all borel probability measures on $C$. The \defi{weak topology} on $\cm$ is the topology generated by the open sets of the form 
$$O_\mu(f_1,\ldots, f_k; \eps_1,\ldots, \eps_k) = \left\{ \nu\in \cm : \lvert \int f_i d\nu - \int f_i d\mu \rvert < \eps_i, 1\leq i \leq k \right\},$$
where $\mu$ ranges over all elements of $\cm$, the $f_i$ range over all bounded continuous functions $f_i: C \to \R$, and the $\eps_i$ range over $\R_{>0}$. An immediate consequence of this definition is that a sequence of measures $\mu_i\in \cm$ converges in this topology to $\mu\in \cm$ \iff\ $\int f d\mu_i$ converges to $\int f d\mu$ \fe\ bounded continuous function $f: C \to \R$. If such a sequence converges, then the limit is unique \cite[Chapter II, Theorem~5.9]{Parth}.

Our main tool in obtaining limits of stochastic processes is the following standard fact, see e.g.\ \cite[Chapter VII, Lemma~2.2]{Parth}.\footnote{Condition (i) in \cite{Parth}[Chapter VII, Lemma~2.2] is void in our case because our spaces have finite diameter.}
\begin{lemma} \label{Parth2.2}
Let $\Gamma$ be a set of probability measures on $C$. Then $\overline{\Gamma}$ is compact \iff\ \fe\ $\eps,\rho>0$ \ti\ $\eta=\eta(\eps,\rho)>0$ \st\ 
$$\mu(\{p \mid \omega_p(\eta)> \rho\}) < \eps \text{ \fe\ } \mu\in\Gamma,$$
where $\omega_p(\eta):= \sup_{|t-t'|\leq \eta} |p(t) - p(t')|$. \note{Konrad, does this \oo\ coincide with the optimal modulus of continuity of $p$?}
\end{lemma}

\note{Konrad, could you extend this from $[0,T]$ to infinity?}

\subsection{Metric graphs and their \BM} \label{Megr}

In this paper, by a \defi{graph} \g we will mean a topological space homeomorhpic to a simplicial 1-complex. We assume that any graph \g is endowed with a fixed homeomorphism $h: K \to \G$ from a simplicial 1-complex $K$, and call the images under $h$ of the 0-simplices of $K$ the \defi{vertices} of \G, and the images under $h$ of the 1-simplices of $K$ the \defi{edges} of \G. Their sets are denoted by \defi{$V(G)$} and \defi{$E(G)$} respectively.  All graphs considered will be \defi{finite}, that is, they will have finitely many vertices and edges.

A metric graph is a graph \g endowed with an assignment of lengths $\ell: E(G) \to \R_{>0}$ to its edges. This assignment naturally induces a metric $d_\ell$ on \g with the following properties. Edges are locally isometric to real intervals, their lengths (i.e.\ 1-dimensional \HM s) \wrt\ $d_\ell$ coincide with $\ell$, and \fe\ $x,y\in V(G)$ we have $d_\ell(x,y):= \inf_{P \text{ is an \arc{x}{y} }} \ell(P)$, where $\ell(P):= \sum_{P \supseteq e \in E(G)} \ell(e)$; see \cite{ltop} for details on $d_\ell$.

The length $\ell(G)$ of a metric graph \g is defined as $\sum_{e\in E(G)} \ell(e)$.

An \defi{interval} of an edge $e$ of \G\ is a connected subspace of $e$.

\medskip
Brownian motion on \R\ extends naturally to Brownian motion on a metric graph. 
The edges incident to a vertex constitute a ``Walsh spider'' (see, e.g., \cite{Wa,BPY}) with equiprobable legs, and it is easily verified that
in such a setting the probability of traversing a particular incident edge (or oriented loop) first is proportional to the reciprocal of the length of that edge, while inside any interval of an edge, it behaves like standard \BM\ on a real interval of the same length. To make this more precise, it is shown in \cite{BaChaEqu} that \ti\ a probability distribution on the space $C(G)$ of continuous functions from a real interval $[0,T]$ to \G, which we will call \defi{standard \BM} on \G, that has the following properties
\begin{enumerate}
\item The \smp;
\item \label{bmii} \fe\ vertex $v$ of \g and any choice of points $p_i, 1\leq i \leq k$, one inside each edge incident with $v$, the probability to reach $p_j$ before any other $p_i, i\neq j$ when starting at $v$ is $\frac{1/\ell_j}{ \sum_{1\leq i \leq k} 1/\ell_i}$, where $\ell_i$ denotes the length of the interval from $v$ to $p_i$ (\cite[\S 4: Lemma 1 applied with $\tilde{p_i}:=1/k$]{BaChaEqu});
\item \fe\ vertex $v$ of \G, the expected time to exit the ball of radius $r$ around $v$ when starting at $v$ tends to 0 as $r$ tends to 0 (\cite[(3.1)]{BaChaEqu}).
\item \label{bmiv} When starting at a point $p$ inside an edge $e$, the expected time till the first traversal of one of the two intervals of $e$ of length $\ell$ starting at $p$ is $\ell^2$ (\cite[(3.4)]{BaChaEqu}).
\end{enumerate}

\comment{
A $k$-subdivision of a metric graph $G$ is a metric graph $G_k$  \st\ every edge of $G_k$ has length smaller than $k$ and \ti\ an isometry $\iota: G_k \to G$. Any \BM\ on $G$ naturally induces a continuous-time random walk $Z_t, t\in \R_+$ on $G_k$, and also a discrete time random walk $R_i, i\in \N$. It follows from \ref{bmii} that the transition probabilities of $Z_t$ and $R_i$ coincide with the transition probabilities of the usual random walk on $G_k$, where the probability to go from a vertex $v$ to each of its neighbours $w$ is $c(vw)/ \sum_{y\tilde v} c(vy)$ if we set $c(vy) =1/\ell(vy)$ \fe\ edge $vy$ incident with $v$. Combined with \ref{bmiv}
this implies that ...

\begin{lemma} \label{Gk}
Let \g be a finite metric graph and $(G_{k_i})_{i\in\N}$ a sequence of subdivisions of \g with $\lim k_i = 0$. Then \fe\ interval $I$ of an edge of \G, w
\end{lemma}
}

\medskip
The expected time for \BM\ started at a vertex $a$ to visit a vertex $z$ and then return to $a$, i.e., $\Ex_a[\tau_z] + \Ex_ z[\tau_a]$, is called the \defi{commute time} between $a$ and $z$.

\begin{lemma}[\cite{CRRST,LyonsBook}]\label{com}
Let \G\ be a finite metric graph, and $a$, $z$ two vertices of \G. The commute time between a and z equals $2\ell(G) R(a,z)$.
\end{lemma}

\subsection{Electrical network basics}

\newcommand{\are}{\vec{e}}
\newcommand{\edge}[2]{\ensuremath{#1}\text{--}\ensuremath{#2}~edge}
\newcommand{\arE}{\vec{E}}
\newcommand{\ded}[1]{\ensuremath{\overrightarrow{#1}}}

An \defi{electrical network} is a graph $G$ endowed with an assignment of resistances $r: E \to \R_+$ to its edges. The set \defi{$\arE$} of \defi{directed edges} of \G\ is the set of ordered pairs $(x,y)$ such that $xy\in E$. Thus any edge $e$ of \g with endvertices $x,y$ corresponds to two elements of $\arE$, which we will denote by $\ded{xy}$ and $\ded{yx}$.
A \defi{\flo{p}{q}} of strength $I$ in \g is a function $i: \arE \to \R$ with the following properties
\begin{enumerate}
 \item $i(\ded{e^0 e^1}) = i(\ded{ e^1e^0 })$ \fe\ $e\in E$ ($i$ is antisymmetric);
\item \fe\ vertex $x\neq p,q$ we have  $\sum_{y\in  N(x)} i(\ded{xy}) = 0$, where \defi{N(x)} denotes the set of vertices sharing an edge with $x$ ($i$ satisfies Kirchhoff's node law outside $p,q$);
\item $\sum_{y\in  N(p)} i(\ded{py}) = I$  and $\sum_{y\in  N(q)} i(\ded{qy}) = -I$ ($i$ satisfies the boundary conditions at $p,q$).
\end{enumerate}

The \defi{effective resistance} $R_G(p,q)$ from a vertex $p$ to a vertex $q$ of \g is 
defined by 
$$R_G(p,q):= \inf_{i \text{ is a \flo{p}{q}\ of strength 1}} E(i),$$
 where the \defi{energy} $E(i)$ of $i$ is defined by $E(i) := \sum_{\ded{e}\in \arE}  i(\ded{e})^2 r(e)$. In fact, it is well-known that  this infimum is attained by a unique  \flo{p}{q}, called the corresponding \defi{electrical current}. 

The effective resistance satisfies the following property which justifies its name

\begin{lemma} \label{effres}
Let \g be an electrical network contained in an electrical network $H$ in such a way that there are exactly two vertices $p,q$ of \g connected to vertices of $H - G$ with edges. Then if $H'$ is obtained from $H$ by replacing $G$ with a \edge{p}{q} of resistance $R_G(p,q)$, then \fe\ two vertices $v,w$ of $H'$ we have $R_{H'}(v,w) =R_{H}(v,w) $. 
\end{lemma}
The proof of this follows easily from the definition of effective resistance. See e.g.\ \cite{LyonsBook} for details. 

\medskip
Any metric graph naturally gives rise to an electrical network by setting $r= \ell$, and we will assume this whenever talking about effective resistances in metric graphs.

The importance of effective resistances for this paper is due to the following fact, showing that they determine transition probablities between any two points in a finite set for \BM\ on a metric graph.

\begin{lemma}[{\cite[Exercise 2.54]{LyonsBook}}] \label{trpr}
Let \g be a metric graph and $U$ a finite set of points of \G. Start \BM\ at a point $o\not\in U$ of \g and stop it upon its first visit to $U$. Then the exit probabilities are determined by the values $\{R(x,y) \mid x,y\in U\cup\sgl{o}\}$.
\end{lemma}

\section{Existence} \label{secCon}

In this section we prove the existence part of \Tr{thmain}, in other words, the existence of an accumulation point in 
$\cm(C)$ of every sequence \seq{B} \st\ $B_n$ is standard \BM\ on a graph $G_n\subseteq X$ and \seq{G} is a \defi{\gseq}:
a \gseq\ of $X$ is a sequence $\seq{G}$ of finite graphs that are subspaces of $X$ satisfying the following two properties:
\begin{enumerate}
\item \fe\ edge $e\in E(G_n)$ the length $\ell(e)$ of $e$ in $G_n$ coincides with the length of the corresponding arc of $X$;
\item \label{comps} {\Fe\ finite edge-set $F$ of $X$, and every component $C$ of $X\sm F$, \ti\ a unique component of $G_i\sm F$ meeting $C$ for almost all $i$.}
\end{enumerate}

The existence of \gseq s was established in \cite{gl}. In fact, we can furthermore assume that each $G_n$ is connected, and that $\gn \subseteq G_{n+1}$ \fe\ $n$, although it will not make a formal difference for our proofs. It is also shown in \cite{gl} that \ref{comps} implies that $\bigcup G_n$ contains every edge of $X$ and is dense in $X$.


So let us fix such a sequence \seq{G}.
\Fe\ $n\in\N$ and $o_n\in G_n$, let $\mu_{n,o}$ be the measure on $C$ corresponding to standard \BM\ on $G_n$ starting at the point $o_n$. Let 
$$\Gamma:= \{\mu_{n,o} \mid n\in\N, o\in G_n\}.$$
The following result shows that this family of measures has accumulation points in $\cm(C)$, which we think of as candidates for our \BM\ on $X$. We will show in \Sr{SeqUniq} that if the $o_n$ converge to a point of $X$, then $\Gamma$ has a unique accumulation point.

\begin{lemma} \label{exist}
The family $\overline{\Gamma}$ is compact (\wrt\ the weak topology). 
\end{lemma}


\begin{proof}

 
Throughout this proof $B_n$ is a random sample path in $C$ chosen according to some of our measures $\mu\in \Gamma$, and probabilities refer to that measure.

We are going to show that our family $\Gamma$ satisfies the condition of \Lr{Parth2.2}, that is, for any $\eps>0$
 \labtequ{cond}{$\lim_{\del\to0}\Pr\left(\sup_{t,s<T;|t-s|<\del}d(B_n(t),B_n(s))>\eps\right)=0$ uniformly in $n$.}

So fix $\eps>0$. Let $R=R_\eps$ be a finite set of  edges as in \Prr{mineps}, and let $\eps_1=\min\{\ell(e) \mid e\in R \}$.

\comment{
  We can split the expression in \eqref{cond} into two parts, according to whether $B_n(t),B_n(s)$ lie in some $U_i$ or not. Note that if $d(B_n(t),B_n(s))>\eps$ then they can not both lie in a common $U_i$, for $\operatorname{diam}(U)<\eps/8$; thus we can write
 \begin{eqnarray*}
  \Pr[\sup_{t,s<T; |t-s|<\del} d(B_n(t),B_n(s))>\eps] \\\le \Pr[d(B_n(t),B_n(s))>\eps\mbox{ for }|t-s|<\del 
 \mbox{ \st } B_n(t)\notin U^+ ]\\
  +\Pr[d(B_n(t),B_n(s))>\eps\mbox{ for }|t-s|<\del 
 \mbox{ \st } B_n(t)\in U_1^+\neq U_2^+\ni B_n(s)]\\
 \end{eqnarray*}

  Note that in both cases, if $d(B_n(t),B_n(s))>\eps$ then the particle has to traverse an edge of $S$: in the first case, i.e.\ when $B_n(t)\notin U^+$, then ... . In the second case, i.e.\ when $B_n(t)\in U_1^+\neq U_2^+\ni B_n(s)$, then $S$ disconnects $U_1^+$ from $U_2^+$ so clearly the particle has to traverse an edge of $S$.
}

Thus we have the following bound for the probability appearing in \eqref{cond}:

\begin{eqnarray*}
\Pr[\sup_{t,s<T; |t-s|<\del} d(B_n(t),B_n(s))>\eps] \le \\ 
\Pr[B_n([t,t+\del])\mbox{ traverses an edge }e\in R\mbox{ for some }t\in[0,T-\del]]. 
\end{eqnarray*}

It remains to show that the last probabilities converge to 0  uniformly in $n$ as $\del \to 0$. For this we will use the fact that each brownian motion $B_n$ in the interior of an edge behaves locally like standard Brownian Motion $W$ on the real line. Let us make this more precise. Let $R'$ be the set of half-edges of $R$, that is, each element of $R'$ is a open subinterval of an edge of $R$ from an endpoint to the midpoint. Let us subdivide the time interval $[0,T]$ into the $\ceil{T/\del}$ subintervals $I_0,I_1,\ldots I_k$ of the form $I_i= [i\del, (i+1)\del]$; note that each $I_i$ has duration at most $\del$. Then, if $B_n$ traverses an edge of $R$ in time $\del$ at some point, then there is a time intervals $I_i$ during which $B_n$ traverses an element of $R'$. Thus we can write

\begin{align*}
 \Pr[B_n([t,t+\del])\mbox{ traverses an edge }e\in R\mbox{ for some }t\in[0,T]]\\
 \le \sum_{i}\Pr[B_n([i\del,(i+1)\del])\mbox{ traverses an edge }e\in R'].
\end{align*}
Now denote by $M$ the set of the midpoints of elements of $R'$, and by $\tau_n^i=\inf\{t\ge i\del:B_n(t)\in M\}$ the associated hitting times. Then we can bound the last expression by
$$
 \sum_{i}\Pr[B_n([\tau_n^i,\tau_n^i+\del])\mbox{ traverses an edge }e\in R''],\\
$$
where $R''$ is the set of half-edges of $R'$, in other words, the `quarter-edges' of $R$.

Now since inside an edge $B_n$ behaves like standard Brownian Motion $W$, the above sum is at most
$$\ceil{{T/ \del}} \Pr[\max_{t\in[0,\del]} |W(t)|>\eps_1/4)=\ceil{{T/ \del}} \Pr[|W(\del)|>\eps_1/4),$$

by reflection principle \cite[Theorem 2.21]{PeresMoerters}. This expression  converges to 0 with $\del$, since the second factor decay rapidly with \del . Moreover, it does not depend on $n$, and so it yields \eqref{cond} as desired.

\end{proof}

Remark: if $\seq{\mu}$ is a  convergent sequence of elements of $\Gamma$ with limit $\mu$, then \fe\ $x\in X$, 
$$\Ex_{\mu}[ d(b(0),x)]  = \lim_n \Ex_{\mu_n}[ d(b(0),x)] .$$
In particular, if the starting points of the $\mu_n$ converge to $x$, then the starting point of $\mu$ is $x$ a.s.

\comment{
	\begin{lemma}\label{ellh}	
Let $H\subseteq G$ be finite graphs, and $a,z\in V(G - H)$. Then the expected time spent in $H$ by SRW in \g during a commute trip between $a$ and $z$ is at most $2\ell(H)\ell(G)$.
\end{lemma}
\begin{proof}
It is known that for SRW on any finite graph \g the expected time of traversals of an edge during a commute trip between any two fixed vertices is the same for any two edges of the graph \cite{CRRST,AKLLR}. Thus, if $C$ is the expected duration of the commute trip, then the expected time spent in $H$ is $C \frac{\ell(H)}{\ell(G)}$. By \Lr{com} we easily have $C\leq 2\ell^2(G)$, yielding the desired bound $2\ell(H)\ell(G)$.
\end{proof}
}

\section{Occupation time of small subgraphs} \label{SecOccu}

A \defi{subgraph} $H$ of a graph \g is a subspace of \g that is a graph itself. If \g is a metric graph, then we consider $H$ to be a metric graph as well, with its edge-lengths induced from those of \g in the obvious way. Note that the vertices of $H$ need not be vertices of \G; an interval of an edge of \g can be an edge of $H$.

For a (finite) metric graph \g and standard \BM\ $B$ on \G, the \defi{occupation time} $OT_t(H)= OT_t(H;B)$ of a subgraph $H\subseteq G$ up to time $t$ is defined to be the amount of time $\int_0^t \mathds 1_{\{B(s)\in H\}} ds$ spent by $B$ in $H$ in the time interval $[0,t]$. We define the occupation time of $H$ for random walk on $G$ similarly.

The following lemma shows that the occupation time of a subgraph $H$ of $G$ is short with high probability when the length $\ell(H)$ is small compared to $\ell(G)$, and in fact can be bounded above by a function depending only on the proportion of the lenghts but not on the structure of $G$ and $H$.
\begin{lemma}\label{ellh2}	
\Fe\ $L,T,\eps\in \R_{>0}$ \ti\ a small enough $\ell\in \R_{>0}$ \st\ \fe\ finite metric graph \g with $\ell(G)\geq L$ and every
subgraph $H\subseteq G$ with $\ell(H)\leq \ell$, we have $OT_T(H)< \eps$ with probability at least $1-\eps$.
\end{lemma}
\begin{proof}
Let $\tau$ be the random time of the first return to the starting point $o$ after time $T$. We claim that
$$ \Ex_B [OT_\tau(H)] = \Ex_B [\tau] \frac{\ell(H)}{\ell(G)},$$
Where the subscript $B$ stands for the fact that the expectation is taken with respect to standard \BM\ on \G.
For this, we use the fact that for simple random walk $R$ on \g  it is well-known \cite{CRRST}  that the expected occupation time $\Ex_R [OT_\tau(K)]$ up to time $\tau$ in any subgraph $K$ equals $\Ex_R [\tau]$ times the stationary distrubution $\pi$ integrated over $K$ (this follows directly from renewal theory \cite[Proposition~7.4.1]{Ross}). That is, we have 
\labtequ{exr}{$ \frac{\Ex_R [OT_\tau(K)]}{\Ex_R [\tau]} = \pi(K) $,} 
where $\pi(K) = \sum_{v\in V(K)} \pi(v)$. 

Now let us assume that all edge lengths of $G$ are rational. Then, we can find a subdivision $G'$ of $G$ \st\ all edges of $G'$ have the same length. Formally, $G'$ is a metric graph isometric to $G$ as a metric space. 
Clearly, we can find subgraphs $H_<,H_>$ of $G$ such that $H_< \subseteq H \subseteq H_>$ and each boundary vertex of $H_<$ or $H_>$ is a midpoint of an edge of $G'$, where a boundary vertex of $H_<$ is one incident with the complement of $H_<$, i.e.\ a point in $H_< \cap \cls{(G \sm H_<)}$. Thus, since the stationary distribution $\pi$ is proportional to the vertex degree, and since every edge of $G'$ has the same length, we have 
\labtequ{piell}{$\pi(H_<) = \frac{\ell(H_<)}{\ell(G)}$ and $\pi(H_>) = \frac{\ell(H_>)}{\ell(G)}$.}

Note that \BM\ $B$ on $G$ naturally induces a continuous-time random walk $Z(t), t\in \R_+$ on $G'$, and also a discrete time random walk $R(i), i\in \N$. It follows from \ref{bmii} in \Sr{Megr} that the transition probabilities of $Z$ and $R$ coincide with the transition probabilities of the usual random walk on $G'$, where the probability to go from a vertex $v$ to each of its neighbours $w$ is $c(vw)/ \sum_{y\tilde v} c(vy)$ if we set $c(vy) =1/\ell(vy)$ \fe\ edge $vy$ incident with $v$.

It is proved in \cite[Section 5.1]{edgecov} that, \fe\ subgraph $K$ of \G, in particular for $K= H_<$ or $K=H_>$, we have
\labtequ{zr}{$ \frac{\Ex_R [OT_\tau(K)]}{\Ex_R [\tau]} =  \frac{\Ex_Z [OT_\tau(K)]}{\Ex_Z [\tau]} .$}

Note that we have $\Ex_Z [\tau] =  \Ex_B [\tau]$ by the definition of the continuous time random walk $Z$. Moreover, using the fact that  each boundary vertex of $H_<$ or $H_>$ is a midpoint of an edge of $G'$, it is possible to prove that 
$$\Ex_Z [OT_\tau(H_<)] = \Ex_B [OT_\tau(H_<)] \text{ and }  \Ex_Z [OT_\tau(H_>)] = \Ex_B [OT_\tau(H_>)]$$ because for each edge $e=xy$ of $G'$, the expected number of traversals of $e$ from $x$ to $y$ up to time $\tau$ equal the expected number of traversals of $e$ from $y$ to $x$ (this follows from the same arguments as in the proof of \ref{exr}), and \BM\ on an interval from an endpoint is equidistributed with its reflectection around the midpoint. Combining this with \eqref{zr}, \eqref{piell} and \eqref{exr}, we obtain
$$ \frac{\Ex_B [OT_\tau(H_<)]}{\Ex_B [\tau]} = \frac{\ell(H_<)}{\ell(G)} \text{ and }   \frac{\Ex_B [OT_\tau(H_>)]}{\Ex_B [\tau]} = \frac{\ell(H_>)}{\ell(G)}.$$

Since $\Ex_B [OT_\tau(H_<)] \leq \Ex_B [OT_\tau(H)] \leq \Ex_B [OT_\tau(H_>)] $ by the choice of $H_<,H_>$, and  $\frac{\ell(H_<)}{\ell(G)},  \frac{\ell(H_>)}{\ell(G)}$ can be made arbitrarily close to $ \frac{\ell(H)}{\ell(G)}$ by making the subdivision $G'$ fine enough, our claim $ \Ex_B [OT_\tau(H)] = \Ex_B [\tau] \frac{\ell(H)}{\ell(G)}$ follows in the case that all edge lengths of $G$ are rational. The general case can now be handled using a standard approximation argument.


\medskip

Thus if $H, G$ are as in the statement, then, since $T \leq \tau$, we obtain
$$\Ex_B [OT_T(H)]\leq \Ex_B [OT_\tau(H)] \leq \Ex_B [\tau] \frac{\ell}{L}.$$
Now if $\Pr[OT_T(H)\geq \eps] > \eps$ then $\Ex_B [OT_T(H)] > \eps^2$. Combined with the above inequality, this yields 
$$\Ex_B [\tau] > \eps^2 \frac{L}{\ell}.$$
On the other hand, applying the commute time formula of \Lr{com} to the pair of points $o, B(T)$ where $B(T)$ is the random position of the particle at time $T$, we obtain $\Ex_B [\tau]  \leq T + 2L^2$ since, easily, $R(a,z) \leq L$ for every two points $a,z$ of $G$. The latter two inequalities imply $T+ 2L^2 \geq \eps^2 \frac{L}{\ell}$, and so letting $\ell =  \frac{\eps^2L}{T+ 2L^2}$ proves our assertion.
\end{proof}

The following lemma is of similar flavour

\begin{lemma} \label{Ldel} 
Let $X$ be a \gl\ continuum with $\Hm(X)< \infty$, and \seq{G} a \gseq\ of $X$. For any time $T_0$ and $p\in  X$ lying in an edge of $X$,  we have   
$$\lim_{r\to0}\sup_n\Pr(B_n(T_0) \text{ is in the ball of radius $r$ centred at }p)=0.$$
\end{lemma} 
\note{Konrad: I adapted the assertion to the new set-up. Please read the proof thoroughly and see if it needs changes.}
\begin{proof}
 The proof of this uses a well-known idea going back to Nash \cite{Nash}, however in order to make it self contained we  present  it now.
 
 Let $P^n_t$ be the heat semigroup associated with the Brownian motion $B_n$ on $G_n$ i.e. $P^n_tf(x)=\Ex_x[f(B_n(t))]$, for any bounded function $f$. By duality $P^n_t$ acts on the space of probability measures on $G_n$. Our assertion will be proven if we  show that the $P^n_{T_0}\delta_y$ have  bounded densities with respect to Hausdorff measure $\Hm$ by some constant independent of $n$ and $y\in G_n$, where $\delta_y$ is the Dirac measure at $y$. Since any $\delta_y$ is a weak limit of a probability measure with a density that is continuous on $G_n$ and differentiable inside every edge, it is sufficient to get a uniform bound on $\|P^n_{T_0}f\|_{\8}$.
 
 The idea (cf. \cite{Nash,Haeseler}) is to prove first a Nash type inequality:
 \begin{align}
  \label{Nash}
   \|u\|_2^{6}\le8(c \|u\|_2^2+\|u'\|_2^2)\|u\|_1^4,
 \end{align}
for every continuous function $u$ which is differentiable inside every edge,  where $c=(2\Hm(G_0))^{-2}$.
 It is enough show \eqref{Nash} for $\|u\|_{1}=1$. Since $u$ is continuous  there is $x_0\in G_n$ such that $|u(x_0)|=1/\Hm(G_n)$. Now for any $x\in G_n$ there is a path $\gamma$ connecting $x$ with $x_0$, and so by the Schwarz inequality we have
 $$ u(x)^2-u(x_0)^{2}=\int_{\gamma}2u(y)u'(y)dy\le2\|u\|_2\|u'\|_2,$$
 which implies
  $$u(x)^2\le\sqrt{\Hm(G_n)^{-2}+2\|u\|_2\|u'\|_2}|u(x)|.$$
  Integrating the above inequality we get
  $$\|u\|_2^2\le \sqrt{\Hm(G_n)^{-2}+2\|u\|_2\|u'\|_2}\le\sqrt{\Hm(G_n)^{-1}\|u\|_2^2+2\|u\|_2\|u'\|_2}.$$
By the inequality between the quadratic and arithmetic mean, this implies that
 $$ \|u\|_2^{6}\le{2\Hm(G_0)^{-2}\|u\|_2^2+8\|u'\|_2^2},$$
and so \eqref{Nash} is proved.
 
 Next, following the idea due to Nash \cite{Nash}, we define $U(t)=\|e^{-\delta t}P_tf\|_2^2$. An easy observation (cf. \cite{Kuchment}) gives that $\frac d{dt}U(t)=-2\delta\|e^{-\delta t}P_tf\|_2^2-2\|e^{-\delta t}(P_tf)'\|_2^2$. In view of \eqref{Nash} this leads to a following inequality:
 $$U^3(t)\le -4\frac d {dt}U(t)e^{-4\delta t}\le-4\frac d {dt}U(t),$$
 since $\|P_tf\|_1=1$.
 By elementary computations $U(t)\le (t/2)^{-1/2}$, hence $\|P_tf\|_2\le e^{\delta t/2}(t/2)^{-1/4} $.  The semigroup principle gives $P_{T_0}=P_{T_0/2}\circ P_{T_0/2}$ and by symmetry $\|P_{T_0/2}\|_{1\to 2}=\|P_{T_0/2}\|_{2\to\8}$ therefore $\|P_{T_0}f\|_{\8}\le e^{\delta T_0/2}(T_0/4)^{-1/2}$.
\end{proof}

\section{Uniqueness} \label{SeqUniq}

\comment{
\begin{lemma}[{\cite[Proposition~9.1]{LyonsBook}}] \label{ign}
Let $N=(G,r,p,q,I)$ be a countable network. Then \fe\ \gseq\  $\gn$ the currents $i(G_n)$ converge in $\ell^\infty(E)$.
\end{lemma}
}

The following fact implies that if $\Hm(X)<\infty$ then the \BM\ we constructed in \Sr{secCon} is uniquely determined by $(X,d)$; 
in particular, it does not depend on the choice of the \gseq\ used. 

\begin{theorem} \label{thuniq}
Let $X$ be a \gl\ space with $\Hm(X)<\infty$. Then \fe\ \gseq\ \seq{G}, and any convergent sequence \seq{o} of points of $X$ with $o_n\in G_n$, standard \BM\ $B^n_{o_n}$ from $o_n$ on $G_n$ converges weakly to an element of $\cm$ independent of the choice of \seq{G}.
\end{theorem}

This follows immediately from the following lemma. The independence of the limit from $(G_n)$ follows from the fact that if $(H_n)$ is another \gseq\ of $X$, then $G_1,H_1,G_2,H_2,\ldots$ is also a \gseq.

\begin{lemma} \label{luniq}
Let $X$ be a \gl\ space with $\Hm(X)<\infty$ and  \seq{G} a \gseq\ of $X$. Let $o_i\in G_i$ be a sequence of points that converges to a point $o\in X$. Then for every finite collection of open sets $A_1,\ldots, A_z$ of \fcg, and every finite collection of time instants $T_1,\ldots T_k \in \R^+$, the probability $\Pr[B_n(T_i) \in A_i \text{ for every } 1\leq i \leq k]$ converges, where $B_n$ denotes standard \BM\ on $G_n$ from $o_i$.
\end{lemma}

The rest of this section is devoted to the proof of \Lr{luniq}. As it is rather involved, we would like to offer the reader the option of reading a simpler proof of a weaker result that still contains many of the ideas: the case where $X$ contains a \des\ $E$ with $\sum_{e\in E} \ell(e) =  \Hm(X)$. 

\readi{The reader choosing this option will be guided throughout the proof as to which parts can be skiped.}

\subsection{Useful facts about \gl\ spaces}

We will be using the following terminology and facts from \cite{gl}.

\begin{theorem}[{\cite{gl}}] \label{trprconv}
Let $X$ be a \gl\ space with \finhm\ and $\seq{G}$ be a \gseq\ of $X$. 
Then \fe\ two sequences \seq{p}, \seq{q} with $p_n,q_n \in G_n$, each converging to a point in $X$, 
the effective resistance $R_{G_n}(p_n,q_n)$ converges. If $p_n=p,q_n=q$ are constant sequences, then this convergence is from above, i.e.\ $\lim_n R_{G_n}(p,q)\leq R_{G_i}(p,q)$ \fe\ $i$.
\end{theorem}

\readi{The reader that chose to read the simplified version can now skip to \Sr{proofstart}.}

\medskip

A \defi{\pe} of a metric space $X$ is an open connected subspace $f$ \st\ $|\partial f|=2$ and no homeomorphic copy of the interval $(0,1)$ contained in $\cls{f}$ contains a point in $\partial f$. We denote the elements of $\partial f$ by $f^0,f^1$, and call them the \defi{endpoints} of $f$. Note that every edge is a \pe. See \cite{gl} for further examples.

We define the \defi{discrepancy} $\delta(f)$ of a \pe\ $f$ by  $\delta(f):=\Hm(f) - d(f^0,f^1)$, which  is always non-negative \cite{gl}.

\begin{theorem}[\cite{gl}] \label{corstruct}
\Fe\ \gl\ continuum $X$ with \finhm, and every $\eps>0$, there is a finite set $\cf$ of pairwise disjoint \pe s of $X$ with the following properties
\begin{enumerate}
\item \label{cfi} $\sum_{f\in \cf} \Hm(f) > \Hm(X)  - \eps$;
\item \label{cfii} $\sum_{f\in \cf} \delta(f)< \eps$;
\item \label{cfiii} $X \sm \cf$ has finitely many components, each of which is clopen in $X \sm \cf$ and contains a point in $\overline{\cf}$;
\item \label{cfiv} \fe\ $f\in \cf$, and every \gseq\ \seq{G},  $G_n \cap \cls{f}$ is connected and contains a \pth{f^0}{f^1} for almost every $n$;
\item \label{cfv} $\cls{\bigcup \cf}$ avoids any prescribed point of $X$;
\item \label{cfvi} $\cf$ contains any prescribed finite edge-set.
\end{enumerate}
\end{theorem}

\comment{
\subsubsection{sketch}
Before giving the formal proof of \Tr{thuniq}, let us sketch the main ideas. When $n<m$ are large, $G_n$ differs from $G_m$ only on a subgraph of very small length compared to $\ell(X)$. This has two useful implications: \\
\begin{enumerate}
\item \label{i} effective resistances between fixed finite sets of points converge with $n$, from which we can deduce that certain transition probabilities converge. 
\item \label{ii} by the results of \Sr{SecOccu}, the occupation time of  $G_m\sm G_n$ by \BM\ is short. 
\end{enumerate}
This suggests the following idea: we construct a process $B^*_n$ on $G_n$ and a process $B^*_m$ on $G_m$ obtained from \BM\ by ignoring its excursions into $X\sm G_m\cap G_n$. We check that $B^*_n$ is very similar to $B^*_m$ for large $n,m$, and use \eqref{ii} to show that they provide a good approximation for the probabilities in question $\Pr[B_n(T_i) \in A_i $.

} 

\subsubsection{Beginning of proof}

\begin{proof}[Proof of \Lr{luniq}] \label{proofstart}

For simplicity we will assume that $k=1$, letting $A_1=:A, T_1=:T$; the same arguments can be used to prove the general case. 
Given an arbitrarily small positive real number \eps, we will find an integer large enough that whenever $n,m$ exceed that integer we have
\labtequ{aim}{ $|\Pr[B_n(T) \in A] - \Pr[B_m(T) \in A]|< \eps$.}
This immediately implies the assertion. So let us fix  $T,A$ and \eps.


Note that it suffices to prove the assertion when $A$ is a \bos\ of $X$. By \Lr{basis} we can assume that the frontier $\partial A$ of $A$ consists of finitely many points, which are  inner points of edges. Thus we can choose $\delta \in \R^+$ small enough that $(A)_\delta:= \bigcup \{Ball_\delta(a) \mid a\in \partial A\}$, where $Ball_\delta(a)$ is the ball of radius \del\ around $a$ in \ltp, is a disjoint union of edges. 

Moreover, by \Cr{Ldel} we can make this $\delta$ \seth\ 
\labtequ{deleps}{\fe\ $n$, we have $\Pr[B_n(T) \in ((A)_\delta)]< \eps/15$.} 
%

%
%

Next, we choose a parameter $\beta$, depending on \del, small enough that it is relatively unlikely that standard \BM\ will traverse one of the intervals in $\partial A_\delta$ in a time interval of length $\beta$. More precisely, denoting by $W(t)$ the standard \BM\ on \R\ starting at the origin, we choose $\beta$ so that
\labtequ{bet}{$\Pr[\max_{t\in [0, \bet]} |B(t)| >\del] < \eps / 15$,}

\subsubsection{Applying the \pe\ structure theorem}

Fix a \gseq\ \seq{G} of $X$ for the rest of this proof.

\readi{The reader that chose to read the simplified version can now skip to \Sr{second}, letting $\cf$ be a finite subset of $E$ with $\sum_{f\in \cf} \Hm(f) > \Hm(X)  - \eps$, assuming $(A)_\delta\subseteq \cf$, letting $\ck$ be the set of components of $X \sm \cf$ ---which is finite \cite[Lemma 2.9]{gl}--- and letting $\Theta:=  \bigcup_{f\in \cf} \partial f =  \bigcup_{K\in \ck} \partial K$. Moreover, almost every $G_n$ contains $\cf$ \cite[Proposition 3.4.]{gl}, hence it also contains $\Theta$ (\ref{thetaae}).  We may assume that $o_n\in \bigcup \ck$  for almost every $n$ \eqref{oinK}, for if $o$ happens to lie in an edge $e$ in $\cf$ we can remove from $\cf$ a sufficiently small subedge of $e$ cointaining $o$, making sure that $o\in \bigcup \ck$ and all the above is still satisfied. By \Tr{trprconv} we have \leth\ $\lim_n |R(p,q) - R_{G_n}(p,q)| =0$ \fe\ $p,q\in \Theta$ (\ref{rth}).
}

Applying \Tr{corstruct} yields a finite set $\cf$ of pairwise disjoint \pe s of $X$ with $\sum_{f\in \cf} \Hm(f) \approx \Hm(X) $ and $\sum_{f\in \cf} \delta(f) \approx 0$. Moreover, $X \sm \cf$ has finitely many components, each of which is clopen in $X \sm \cf$ and contains a point in $\overline{\cf}$. Let $\ck$ be the set of these components. We can also assume by \Tr{corstruct} that \fe\ $f\in \cf$, the graph $G^f_n:=G_n \cap \cls{f}$ is connected and contains a \pth{f^0}{f^1} for almost every $n$. Moreover, we can assume by \ref{cfvi} that  $(A)_\delta\subseteq \cf$. Applying \ref{cfv} to $o$ we can assume that $\cf$ avoids an open neighbourhood of $o$, and hence
\labtequ{oinK}{$o_n\in \bigcup \ck$  for almost every $n$.}

Note that \fe\ component $K\in \ck$, we have $\partial K \subseteq \bigcup_{f\in \cf} \partial f$ because $K$ is clopen in $X \sm \cf$ and $f$ is open in $X$. Thus we can write $\Theta:=  \bigcup_{f\in \cf} \partial f =  \bigcup_{K\in \ck} \partial K$. Since we know that the subgraph $G^f_n$ of $G_n$ contains a \pth{f^0}{f^1} for almost every $n$, it follows that 
\labtequ{thetaae}{$G_n$ contains $\Theta$ for almost every $n$.}

It follows easily from the definitions that \fe\ $f\in \cf$, the sequence of graphs  \seq{G^f}\ is a \gseq\ of $\cls{f}$. Thus we can apply \Tr{trprconv} to this sequence to deduce that their effective resistances $R_{G^f_n}(f^0,f^1)$ converge to a value that we will denote by $R_f$.

By \Tr{trprconv} again, the effective resistances $R_{G_n}(p,q)$ between any two points $p,q$ in the boundary $\partial K$ of some component $K\in \ck$ converge with $n$ from above to a value that we will denote by $R(p,q)$ (where we used \eqref{thetaae}). Thus we have

\begin{enumerate}
\item \label{rf} $\lim_n |R_f - R_{G^f_n}(f^0,f^1)| = 0$ \fe\ $f\in \cf$, and
\item \label{rth} $\lim_n |R(p,q) - R_{G_n}(p,q)| = 0$ \fe\ $p,q\in \Theta$.
\end{enumerate}

\subsubsection{The first coupling}

The first step in our proof will be to couple our \BM\ $B_n$ on $G_n$ with standard \BM\ $B^-_n$ on a simplified version $G^-_n$ of $G_n$, which can be thought of as being obtained from $G_n$ by turning the \pe s in $\cf$ into edges. 

\readi{This step can be omitted if $\cf$ are edges to begin with, and the reader who chose to read the simplified version of this proof can skip the rest of this subsection.}

\medskip
Let $G^-_n$ denote the graph obtained from $G_n$ by replacing, \fe\ $f\in \cf$, the subgraph $G^f_n= G_n \cap \cls{f}$ with an edge $e_f$ with endvertices $f^0,f^1$ and length $\ell(e_f)= R_{G^f_n}(f^0,f^1)$. Recall that by \eqref{oinK}, $o_n \not\in \bigcup \cf$. 
If $f$ happens to be an edge to begin with, then it remains an edge of $G^-_n$; in particular, $(A)_\delta$ is still contained in the set of edges of $G^-_n$.
Let $B^-_n$ denote standard \BM\ from $o_n$ on $G^-_n$.

In order to couple $B_n$ with $B^-_n$, we are going to modify $G_n$ into $G^-_n$ in a more elaborate way than described above, using more local changes. 

For this, choose some $f\in \cf$, and recall that, by the definition of a \pe, and by \ref{cfiv}, $G^f_n$ is connected, it contains a \arc{f^0}{f^1}\ $P$, and both $f^0,f^1$ have degree 1 in $G^f_n$. We can choose $P$ to be the shortest such arc; this is easy to do since  $G^f_n$ is a finite graph and so there are only finitely many candidates. 

We claim that \ti\ a finite edge-set $\cp$ (in the topological sense of \Sr{sgl}) contained in $P$, \st\ letting
$\cc$ denote the set of components of $G^f_n - \cp$, and letting
$\Pi$ denote the finite set $\partial \cp \sm \{f^0,f^1\}$ separating $\cp$ from $\cc$, we have (see top half of \fig{pe})

\begin{enumerate}
\item \label{cci} No $C\in \cc$ contains $f^0$ or $f^1$; 
\item \label{ccii} Each $C\in \cc$ contains at most 2 elements of $\Pi$, and
\item \label{cciv} $\sum_{C\in \cc} \ell(C) \leq 2 \ell(G^f_n\sm P) \approx 0$.
\end{enumerate}

To show this, \fe\ component $K$ of $G^f_n\sm P$ we let $P(K)$ denote the minimum subpath of $P$ separating $K$ from $G^f_n\sm K$; thus $K$ sends at least one edge to each endvertex of $P(K)$ by its minimality. Note that $P(K)$ is trivial, i.e.\ just a vertex, if that vertex alone separates $K$.
Let $B$ denote the union of the $B(K)$ over all such components $K$. Note that $B$ is a disjoint union of subpahts of $P$, some of which might be the union of several intersecting $B(K)$. Let $\cp$ be its complement $P\sm B$, and let $\Pi=\partial \cp \sm \{f^0,f^1\}$ be the set of endvertices of these paths.

It is clear that this choice satisfies \ref{cci}, since none of the components $K$ above send an edge to $f^0$ or $f^1$ because, since $f$ is a \pe\ and $G^f_n$ is contained in it, each of these vertices has only one incident edge, and that edge must be in $P$.

To see that  \ref{ccii} is satisfied, suppose $C$ contains 3 vertices $x,y,z\in \Pi$ lying in that order on $P$,  let $e$ be an edge in $\cp$ incident with $y$, and let $R$ be an \arc{x}{z}\ in $C$. Let $x'$ be the last point on $R$ in the component of $P\sm e$ containing $x$, and $z'$ the first   point on $R$ in the component of $P\sm e$ containing $z$. Then the subarc of $R$ from $x'$ to $z'$ avoids $P$ and hence shows that $e$ is contained in $B$. This contradicts our choice of $\cp$, and proves  \ref{ccii}.


Finally, \ref{cciv} is tantamount to saying that the subgraph $P\sm \cp$ of $P$ contained in $\bigcup \cc$ has length at most $\ell(G^f_n\sm P)$. This follows from our choice of $P$ as a shortest \arc{f^0}{f^1}: for if we contract each component $K$ of $G^f_n\sm P$ together with $P(K)$ (as defined above), then we are left with a path of length $\ell(\cp)$ at the end, and for each contracted subarc $R$ of $P$ we have contracted a subgraph of $G^f_n\sm P$ of length at least $\ell(R)$.
\medskip
\epsfxsize=\hsize

\showFig{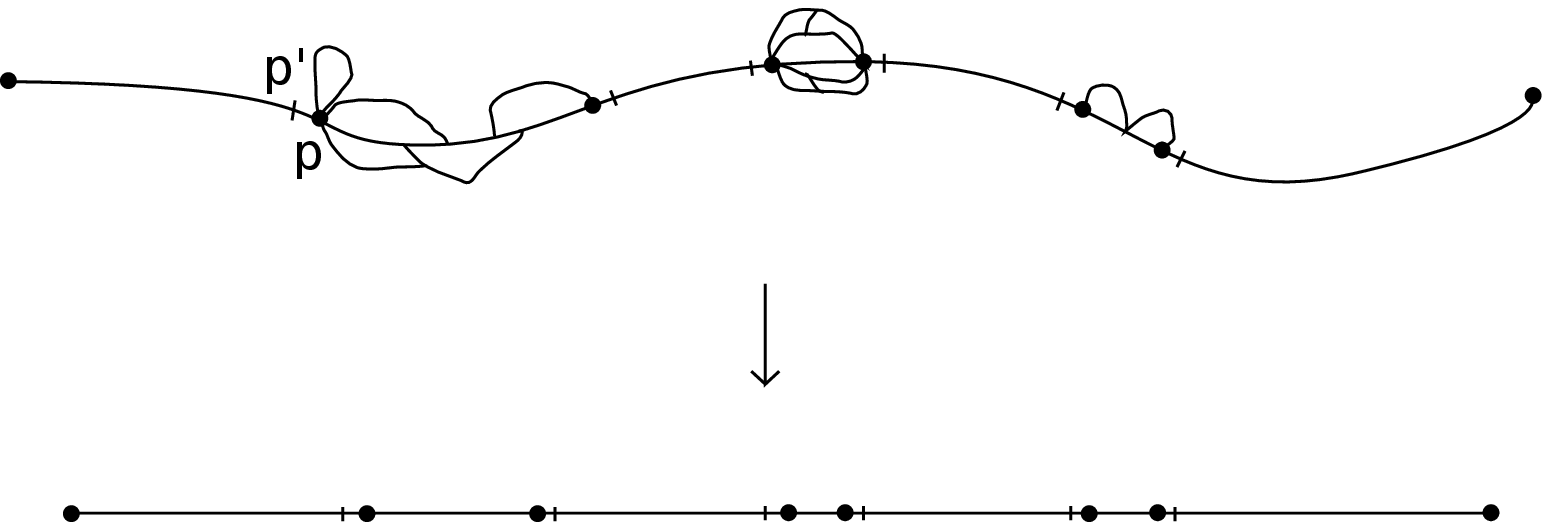}{Replacing the components $\cc$ of $G^f_n - \cp$ with equivalent edges.}
Now replace each component $C\in \cc$ containing two elements $v,w$ of $\Pi$ with a $v$-$w$~edge of length $R_{C}(v,w)$ (\fig{pe}). Then contract any $C\in \cc$ that contains only one element of $\Pi$ into that point. Note that this modifies $G^f_n$ into a \arc{f^0}{f^1} $P'$. 

Note that $R_{P'}(f^0,f^1) = R_{G^f_n}(f^0,f^1)$ by \Lr{effres}. Since we chose $\ell(e_f)= R_{G^f_n}(f^0,f^1)$ in the above definition of $G^-_n$, it follows that if we perform these modifications on each $f\in \cf$ then the resulting graph will be isometric to $G^-_n$.

\medskip
In order to couple $B_n$ with \BM\ $B^-_n$ on $G^-_n$, we pick a set of points $\Pi'$ on $P$ as follows. By definition, every $p\in \Pi$ is incident with exactly one element $R_p$ of $\cp$, which is a subpath of $P$. We choose a point  $p'$ on $R_p$ that is very close to $p$; more precisely, we choose these points $p'$ in such a way that, letting $r_p$ denote the subarc of $R_p$ between $p$ and $p'$, we have
\labtequ{pp}{$\sum_{p\in \Pi} \ell(r_p) <\ell(G^f_n\sm P)$.}
Since we can choose the $p'$ as close as we wich to $p$, there is no difficulty in satisfying this.

In order to perform the desired coupling, we separate the sample path of $B_n$ into excursions by stopping at first visit to $\Pi$, then at the first visit to $\Pi \cup \Pi'$ thereafter (there are always 2 candidate points at which we can stop, one in $\Pi$ and one in $\Pi'$), then at the next visit to $\Pi$, and so on. To couple with \BM\ on $G^-_n$, replace each such excursion $R$ starting at a point $p$ in $\Pi$ by an excursion on $G^-_n$ with same starting point $p$ and stopping upon its first visit to $\Pi \cup \Pi'$ (again, there are 2 candidate points at which we can stop), conditioned on stopping at the same point where $R$ stopped.

Since transition probabilities are the same by \Lr{trpr}, the resulting process is equidistributed with \BM\ $B^-_n$ on $G^-_n$. The two graphs differ in that $\bigcup \cc$ is replaced by edges. The coupling is such that the two processes only differ as to the time they spend in $\bigcup \cc \cup \bigcup_{p\in \Pi} r_p$ or the part of $G^-_n$ that replaces it respectively. We will use \Lr{ellh2} to bound this time. 

We claim that $B_n$ behaves similarly to $B^-_n$ \wrt\ our open set $A$; more precisely, we claim that 
\labtequ{bothA}{$\Pr[\{B^-_n(T) \in A \text{ and }  B_n(T) \not\in A\}\text{ or } \{B^-_n(T) \not\in A \text{ and }  B_n(T) \in A\}] < \eps/5$.}
To prove this, suppose that the event appearing in \eqref{bothA} occured. Recall that the two graphs $G_n,G^-_n$ differ in that $\bigcup \cc$ is replaced by a set of edges $E_\cc$. Let $\Delta_1:= OT_{2T}(\bigcup \cc)$ and $\Delta_2:= OT^-_{2T}(E_\cc)$ denote the occupation time of this difference $\bigcup \cc$ or $E_\cc$ by $B_n$ and $B^-_n$ respectively up to time $2T$ (the reason for the factor 2 will become apparent below).
We claim that in this case, at least one of the following (unlikely) events occured as well:
\begin{enumerate}
 \item \label{Ai} $\Delta_1 > \beta$ or $\Delta_2 > \beta$ (large occupation time of a small set);
 \item \label{Aii} $B_n(T)$ or $B^-_n(T)$ is in $(A)_\delta$ (particle in a small set at time $T$);
 \item \label{Aiii} $B_n([T,T+ \beta])$ or $B^-_n([T ,T+ \beta])$ crosses an edge in $(A)_\delta$ (fast crossing of an edge).
\end{enumerate}
To see this, let $\tau_i$ denote the time $t$ that $B_n(t)$ has just crossed $(A)_\delta$ for the $i$th time; thus $B_n(\tau_i)$ is an endpoint of $(A)_\delta$, and $B_n[t,\tau_i]$ is contained in some edge in $(A)_\delta$ for sufficiently large $t$. Define $\tau^-_i$ similarly for $B^-_n(t)$. Note that if $o_n\in A$, then $B_n(\tau_{2i+1})\in A^c$ and $B_n(\tau_{2i})\in A$ \fe\ $i\in \N^*$ since $ (A)_\delta$ separates $A$ from its complement $A^c$, and so in order to `change sides' from $A$ to $A^c$ the particle has to cross $(A)_\delta$.

Let $k$ denote the largest integer \st\ $\tau_k<T$, and $m$ the largest integer \st\ $\tau^-_k<T$; since $ (A)_\delta$ is a finite edge-set, these numbers are well-defined since $B_n[0,T]\in C$ is continuous and can therefore only cross $ (A)_\delta$ finitely often.
Now if the event appearing in \eqref{bothA} occured, but \ref{Aii} did not, then $k\neq m$. Suppose that $k> m$; the other case is similar. This means that $\tau_k<T$ and $\tau^-_k\geq T$.

Let us assume \obda\ that $T<\beta$, which we can because we can choose \bet\ as small as we wish. It is not hard to see that, unless \ref{Ai} occured, $\tau^-_k < 2T$ holds, since the two processes only differ in their excursions inside $\bigcup \cc$ or $\bigcup \cc$, and their duration yields a bund on how much $\tau^-_k$ can differ from $\tau_k$.  

Note that $\tau^-_k- \tau_k \leq OT_{\tau^-_k}(E_\cc) - OT_{\tau_k}(\bigcup \cc) \leq OT_{2T}(E_\cc) - OT_{\tau_k}(\bigcup \cc)$ by the above argument. Thus if the event \ref{Ai} did not occur, then  $\tau^-_k - T \leq \beta$ holds since $\tau_k<T$. Since $B^-_n(\tau^-_k)$ has just crossed $(A)_\delta$, this means that either $B^-_n(T)$ is in $(A)_\delta$, or $B^-_n([T,T + \beta])$ traversed an edge in $(A)_\delta$; but this is event \ref{Aii} or \ref{Aiii} respectively.

This proves our claim that the event appearing in \eqref{bothA} implies one of the above events.
The probability of each of these 3 events can be shown to be less than $\eps/15$: firstly, by \Lr{ellh2}, and by  \ref{cciv}  and \eqref{pp}, given $L,T,\eps$ and $\bet$ we can 
 make the expectation of $\Delta_1$ and $\Delta_2$ arbitrarily small if we can make $\ell(G^f_n\sm P)$ small enough. We can make the latter arbitrarily small indeed because it is bounded from above by the discrepancy $\delta(f)$ of $f$, which we can make arbitrarily small by \ref{cfii} in \Tr{corstruct}; here, we use the fact that $ \ell(G^f_n)\leq \Hm(f)$ and $\ell(P)\geq d(f^0,f^1)$. Thus the probability of \ref{Ai} can be made less than $\eps/15$.

Secondly, \eqref{deleps} shows that the probability of \ref{Aii} is less than $\eps/15$ as well. Finally, the choice of \bet\ (recall \eqref{bet}) makes \ref{Aiii} equally unlikely. This completes the proof of \eqref{bothA}, which implies in particular
\labtequ{bothAp}{$|\Pr[B_n(T) \in A] - \Pr[B^-_n(T)\in A]| < \eps/5$.}

\subsubsection{The second coupling} \label{second}

\readi{The reader who chose to read the simplified version can assume that $G^-_n=G_n$ and $B^-_n= B_n$. This reader will also need the following definitions. Let $\Theta':=\Theta$ and $e'_f=e_f=f$. For each point $p\in \Theta$, choose a further point $p''$ inside $f$ that is close to $p$ (\fig{thetas}); more precisely, we choose these points $p''$ in such a way that, letting $e_p$ be the interval of $f$ between $p$ and $p''$, we have
$\sum_{p\in \Theta} \ell(e_p) \approx 0$ \eqref{ppp}.
Let also $e_{p'}= e_p$ and skip to \Dr{def}.}


In this section we will couple the processes $B^-_n$ with jump process $B^*_n$, which we will later show that can be coupled between them for various values of $n$.

Recall that the effective resistance $R_{G^f_n}(f^0,f^1)$, which we assigned to each edge $e_f$ as its length $\ell(e_f)$, converges to a value $R_f$ from above. Thus \fe\ such edge $e_f, f\in \cf$, we can choose an interval $e'_f$ with length $\ell(e'_f)= R_f$ independent of $n$. 

Let $\Theta'$ denote the set of endpoints $\partial \bigcup_{f\in \cf} e'_f$ of these edges, and note that each point $p'\in \Theta'$ is close to a point $p\in \Theta$ by \ref{rf}; more precisely, letting $e_p$ be the interval of $e_f$ between $p$ and $p'$, we have
\labtequ{ppp}{$\sum_{p\in \Theta} \ell(e_{p}) <h$,}
where $h=h(\eps,T)$ is a parameter that we can choose to be as small as wish by choosing $n$ large enough.

For each such point $p'\in \partial e'_f$ we choose a further point $p''$ inside $e_f$ that is close to $p'$ (\fig{thetas}); more precisely, we choose these points $p''$ in such a way that, letting $e_{p'}$ be the interval of $e_f$ between $p'$ and $p''$, we have
\labtequ{pppp}{$\sum_{p'\in \Theta} \ell(e_{p'}) < h$.}
\epsfxsize=.7\hsize

\showFig{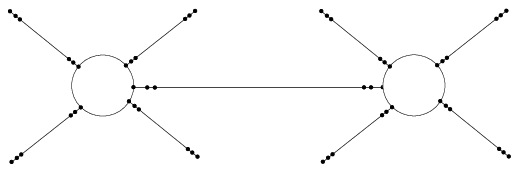}{The sets $\Theta,\Theta'$ and $\Theta''$ around two components in $\ck$.}

Let $\Theta'':= \{p'' \mid p' \in \Theta'\}$. We will use the points in $\Theta'$ and $\Theta''$ similarly to the sets $\Pi, \Pi'$ in the previous section to produce a new process $B^*_n$ coupled with $B^-_n$.

\begin{definition} \label{def}
Let $G^*_n$ be the metric graph obtained from $G^-_n$ by contracting each component of $G^-_n\sm \Theta'$ containing an element $K$ of $\ck$ ---recall that this was the (finite) set of components of $X \sm \cf$--- into a vertex $v_K$. 
\end{definition}
Thus each contracted set comprises a $K\in \ck$ and a short subedge of each edge of $G^-_n$ incident with $K$.

Note that $G^*_n$ is isometric to $G^*_m$ for $n,m$ large enough, because $\ck$ and $\cf$ are fixed and so are the lengths of the edges $e'_f$ of $G^*_n$. We can thus denote by $G^*$ a metric graph isometric to all  $G^*_n$, and let $\iota_n:G^*_n \to G^*$ be the corresponding isometry. Moreover, if $f\in \cf$ happens to be an edge, e.g.\ one of the edges in $(A)_\delta$, then we have $e'_f=f$ in the above definition; this means that $\iota_n(\partial A)= \iota_m(\partial A)$.


We now modify $B^-_n(t)$ into a jump process $B^*_n(t)$ on $G^-_n$, that can also be thought of as a jump process on $G^*_n$. The jumps are always performed from $\Theta'\cup \sgl{o_n}$ to $\Theta''$ and are quite local, so that $B^*_n(t)$ is similar to $B^-_n(t)$. The advantage of $B^*_n(t)$ is that we can couple these processes for various values of $n$ more easily, since they can be projected to the fixed graph $G^*$ via $\iota_n$. Moreover, it will turn out that the event we are interted in, namely whether $B^-_n(T)$ lies in $A$ or not, is tantamount to the projected particle being in the right side of $G^*\sm \iota(\partial A)$.

\medskip
To obtain $B^*_n(t)$ from $B^-_n(t)$, we first sample the path of the latter, then we go through  this path and each time we visit a point $x$ in $\Theta'$, we jump from $x$ directly to the first point $y$ in $\Theta''$ visited afterwards, removing the corresponding time interval from the domain of $B^-_n(t)$ to obtain $B^*_n(t)$ (at the time instant $t$ where this jump occurs we set $B^*_n(t)=y$, say, so that $B^*_n(t)=y$ is right-continuous). 

Recall that $o_n\in \bigcup \ck$ \fe\ $n$ \eqref{oinK}.
When constructing $B^*_n(t)$ from $B^-_n(t)$, we thus also jump over the initial subpath of $B^-_n(t)$ from $o_n$ to the first point $y$ in $\Theta''$ visited, so that $B^*_n(0) = y\in \Theta''$.

Note that $\Theta'$ and $\Theta''$ are finite sets, whence closed in $X$, and so for any topological path (like $B^-_n(t)$) the first visit to any of them is well-defined by elementary topology. Note moreover that we have only finitely many such jumps in the time interval $[0,T]$ because $B^-_n$ is continuous. 

As mentioned above, $B^*_n(t)$ can be thought of as a jump process on $G^*_n$ or $G^*$; the jumps occur whenever a vertex of $G^*$ is visited, and lead to a nearby point of an edge incident with that vertex. From then on, the process behaves like standard \BM\ untill the next visit to a vertex. We will use \Lr{ellh2} to show that the time intervals jumped by $B^*_n(t)$ are relatively short, and so the two processes $B^-_n(t)$ and $B^*_n(t)$ are very similar.


\medskip
\subsubsection{The jump process $B^*_n$ is similar to $B^-_n$}

Our next aim is to show that $B^-_n$ behaves similarly to $B^*_n$ \wrt\ our open set $A$; more precisely, we claim that 
\labtequ{bothA2}{$\Pr[\{B^*_n(T) \in A \text{ and }  B^-_n(T) \not\in A\}\text{ or } \{B^*_n(T) \not\in A \text{ and }  B^-_n(T) \in A\}] < \eps/5$.} 

The proof of this is almost identical to the proof of \eqref{bothA}, but we will reproduce it for the convenience of the reader.

Let $\Delta$ denote the total duration of the intervals `jumped' by $B^*_n$ in the time interval $[0,2T]$.
In order for the event appearing in \eqref{bothA2} to occur, at least one of the following events must occur:
\begin{enumerate}
 \item \label{A2i} $\Delta > \beta$;
 \item \label{A2ii} $B^-_n(T)$ is in $(A)_\delta$;
 \item \label{A2iii} $B^-_n([T,T + \beta])$ traverses an edge in $(A)_\delta$.
\end{enumerate}
To see this, let $\tau_i$ denote the time $t$ that $B^-_n(t)$ has just crossed $(A)_\delta$ for the $i$th time; thus $B^-_n(\tau_i)$ is an endpoint of $(A)_\delta$. Define $\tau^*_i$ similarly for $B^*_n(t)$. Again, if $o_n\in A$, then $B^-_n(\tau_{2i+1})\in A^c$ and $B^-_n(\tau_{2i})\in A$ \fe\ $i\in \N^*$ since $ (A)_\delta$ separates $A$ from its complement $A^c$, and so in order to `change sides' from $A$ to $A^c$ the particle has to cross $(A)_\delta$.

Let $k$ denote the largest integer \st\ $\tau_k<T$, and $m$ the largest integer \st\ $\tau^*_{m}<T$; since $ (A)_\delta$ is a finite edge-set, these numbers are well-defined since $B^-_n[0,T]\in C$ is continuous and can therefore only cross $ (A)_\delta$ finitely often.
Now if the event appearing in \eqref{bothA2} occured, but \ref{A2ii} did not, then $k\neq m$, hence $k< m$ since $B^*_n(t)$ is by definition faster than $B^-_n(t)$. This means that $\tau_m\geq T$ although $\tau^*_{m}< T$.

Let $Y:= \bigcup \ck \cup \bigcup_{p\in \Theta} (e_p \cup e_{p'})$, and recall that this is the subgraph of $G^-$ inside which $B^*_n(t)$ performs its jumps. Let us assume \obda\ that $T<\beta$, which we can because we can choose \bet\ as small as we wish. It is not hard to see that, unless \ref{A2i} occured, $\tau_m < 2T$ holds, since $\tau^*_{m}<T$ and the duration of the excursions inside $Y$ yields a bound on how much $\tau_m$ can differ from $\tau^*_m$.

Now note that $\tau_m- \tau^*_{m} \leq OT_{\tau_m}(Y; B^-_n)$. Thus if the event \ref{A2i} did not occur, then  $\tau_{m} - T \leq \beta$ holds since $\tau^*_{m}<T$. Since $B^-_n(\tau_{m})$ has just crossed $(A)_\delta$, this means that either $B^-_n(T)$ is in $(A)_\delta$, or $B^-_n([T,T + \beta])$ traversed an edge in $(A)_\delta$; but this is event \ref{A2ii} or \ref{A2iii} respectively.

This proves our claim that the event appearing in \eqref{bothA2} implies one of the above events.
The probability of each of these 3 events can be shown to be less than $\eps/15$: firstly, by \Lr{ellh2}, given $L,T,\eps$ and $\bet$ we can 
 make the expectation of $\Delta$ arbitrarily small if we can make $\ell(\bigcup \ck \cup \bigcup_{p\in \Theta} (e_{p} \cup e_{p'}))$ small enough. We can make the latter arbitrarily small indeed by \eqref{ppp}, \eqref{pppp} and by \ref{cfi} in \Tr{corstruct} since $\bigcup \ck$ is the complement of $\cf$. Thus the probability of \ref{A2i} can be made less than $\eps/15$.
Secondly, \eqref{deleps} shows that the probability of \ref{A2ii} is bounded by $\eps/15$ as well. Finally, the choice of \bet\ (recall \eqref{bet}) makes \ref{A2iii} equally unlikely. This completes the proof of \eqref{bothA2}, which implies in particular
\labtequ{bothA2p}{$|\Pr[B^-_n(T) \in A] - \Pr[B^*_n(T)\in A]| < \eps/5$.}

\subsubsection{$B^*_n$ is similar to $B^*_m$ for $n,m$ large; the last coupling}

We have thus shown that $\Pr[\{B^-_n(T) \in A]$ is very close to $\Pr[B^*_n(T) \in A]$. It remains to show that the dependence of the latter on $n$ can be ignored: we claim that
\labtequ{prz}{$|\Pr(B^*_n(D) \in A) - \Pr(B^*_m(D) \in A)| < \eps/5$.}
Combined with \eqref{bothAp} (\readi{which the reader of the simpler version can take for trivially true}) and \eqref{bothA2p}, this would imply \eqref{aim}.

For this, we would first like to bound the number of times that $B^*_n(T)$ commutes between $\Theta''$ and $\Theta'$. But this is easy to achieve:
Let $r:= min_{p' \in \Theta'} \ell(e_{p'})$. 
We claim that \ti\ a constant $M=M(r)$ \leth\ the probability that $B^-_n$ commutes between $\Theta''$ and $\Theta'$ more than $M$ times in the time interval $[0,T]$ is $< \eps/15 $. Indeed, as $r>0$, there is a positive probability $q$, depending only on $r$, that the time it takes $B^-_n$ to traverse any of the edges $e_{p'}, p'\in \Theta'$ is at least $T$. Since any commute between $\Theta''$ and $\Theta'$ involves such a traversal, commuting between $\Theta''$ and $\Theta'$ more than $M$ times in the time interval $[0,T]$ thus happens with probability at most $(1-q)^M$. Choosing $M$ large enough we can make this probability as small as we wish. 
As $\Delta$ is probably small (see previous section), we may assume that  the probability that $B^*_n([0,T])$ commutes between $\Theta''$ and $\Theta'$ more than $2M$ times is also less than $\eps/15$.

For the proof of \eqref{prz} it is useful to considered $B^*_n$, or rather $\iota_n(B^*_n)$, as a jump process on $G^*$, for then $B^*_n$ and $B^*_m$ take place on the `same' metric graph and are easier to couple. To achieve this coupling, we first construct a more convenient realisation of $B^-_n$ as follows. Pick \fe\ $p\in \Theta''$ a sequence $C^p_{n,1}(t), C^p_{n,2}(t) \ldots$ of i.i.d.\ sample paths of \BM\ on $G^-_n$, each distributed like $B^-_n(t)$ starting from $p$ and stopping upon their first visit to $\Theta'$. Similarly, pick \fe\ $q\in \Theta'$  a sequence $D^q_{n,1}(t), D^p_{n,2}(t) \ldots$ of i.i.d.\ sample paths of \BM\ on $G^-_n$, each distributed like $B^-_n(t)$ starting from $q$ and stopping upon their first visit to $\Theta''$. These sample paths can be glued together to produce a path distributed identically to $B^-_n(t)$: start a \BM\ at $o_n$, and stop it upon its first visit to a point $q$ in $\Theta''$. Append to this random path the path $C^p_{n,1}$. If the last point visited by the latter is $q$, then append $D^q_{n,1}$. Continue like this, appending paths of the form $C^p_{n,i}$ and $D^q_{n,j}$ alternatingly, each time choosing the right $p$ or $q$ and the smallest $i$ or $j$ for which the path  $C^p_{n,i}$ or $D^q_{n,j}$ has not been used yet. As \BM\ has the Markov property, the random path thus obtained has indeed the same distribution as $B^-_n$.

The advantage of this realisation of $B^-_n$ is that the paths $C^p_{n,i}$ can be coupled with the $C^p_{m,i}$ \fe\ $n,m$. Now note that by construction, the process $B^*_n$ is obtained from $B^-_n$ by discarding all the $D^q_{n,i}$ in the above construction, 
as well as the initial path from $o_n$ to the first visit to $\Theta''$.

This means that another realisation of $B^*_n$ can be constructed directly by concatenating random paths of the form $C^p_{n,i}$ rather than first constructing $B^-_n$ as above, and then discarding some of its subpaths. For this, we choose a random starting point $p\in \Theta''$  according to the distribution $P^{o}_n$ of the first point in $\Theta''$ visited by \BM\ from $o_n$ in $G^-_n$, and use the path $C^p_{n,1}$. Then we recursively concatenate this path with further paths of this form. In order to decide which path $C^p_{n,i}$ to use next, let $q\in \Theta'$ be the last point visited by the last such path used, choose  a random $p\in\Theta''$ according to the distribution $P^q_n$ of the first point in $\Theta''$ visited by \BM\ from $q$ on $G^-_n$, and use $C^p_{n,i}$ for the least $i$ for which this path has not been used yet to extend the path obtained so far. 

The probability distributions $P^q_n, q\in \Theta' \cup \{o_n\}$ used above depend little on $n$: note that $P^q_n$ and $P^q_m$ have the same finite domain $\Theta''$. By Lemmas~\ref{trpr} and \ref{trprconv}, these distributions converge. This means that we can couple the experiments of choosing one point in $\Theta''$ according to $P^q_n$ and one according to $P^q_m$ in such a way that the probability that the two experiments yield a different point is smaller than $\eps/15(2M)$, say, if $n,m$ are sufficiently large (this remains true if $q=o_n$ in the first case and $q=o_m$ in the second).

Combining this coupling with that of the $C^p_{n,i}$, we deduce that $B^*_n$ can be coupled with $B^*_m$ in such a way that they coincide up to the first time that they jump to a distinct element of $\Theta''$, an event occuring with probability smaller than $\eps/15(2M)$ each time that a jump is made. The choice of $M$ now implies that $B^*_n$ coincides with $B^*_m$ up to time $T$ with probability at least $1- \eps/10$ when the processes are so coupled. This proves \eqref{prz}.

Combining this with \eqref{bothAp} and \eqref{bothA2p}, each applied once for $l=n$ and once for $l=m$, yields $|\Pr(B^-_n(T) \in A) - \Pr(B^-_m(T) \in A)|< < \eps$, and so $\Pr(B^-_n(T) \in A) $ converges indeed.
\end{proof}

\section {Strong Markov Property} \label{ssmp}
By the previous section we know that for any open $A$ in $G$ and $x_n\to x$ the probabilities $\p{x_n}{B_n(t)\in A}$ converge to $\p{x}{B(t)\in A}$. 
\note{Is it true? It is crucial for the rest\\ YES!}For any continuous function $f$ on $G$, by  portmanteau theorem\note{pls provide reference}, $P^n_tf(x_n)$ also converge to $P_tf(x):=\e{x}{f(B(t))}$  where $P^n_tf(y)=\e{y}{f(B_n(t))}$ for $y\in G_n$ and it is extended by 0 to $G$.

The strong Markov property follows  by similar methods as in \cite{Barlow1989}.  We start with elementary lemma
\begin{lemma}
Suppose $f$ and $f_n$ are functions on $G$ with the property that  $f_n(x_n) \to f (x)$ whenever  $x_n\in G_n$, $x_n\to x$. Then  $f$ is continuous and $$\sup_{y\in G_n}|f_n(y)- f(y)|\to0.$$
\end{lemma}
\begin{proof}
 In order to prove continuity observe that, by density of $\bigcup G_n$ in $G$, it is enough that show that for $x_n\in G_n$, $x_n\to x$ we have $f(x_n)\to f(x)$. Since $f_m(x_n)\to f(x_n)$, we can take an increasing sequence $m_n$ such that $f_{m_n}(x_n)- f(x_n)$ goes to zero. Since  $f_{m_n}(x_n)$ is a subsequence of $f_k(x'_k)$, where $x'_k=x_n$ when $k\in [m_n, m_{n+1})$, $f_{m_n}(x_n)\to f(x)$. This gives that $f$ is continuous.
 
 Suppose that the second part of the theorem fails. Then we have a subsequence $n_k$ and $x_{n_k}\to x$ with $|f_{n_k}(x_{n_k})-f(x_{n_k})|>\epsilon$ for some $\epsilon>0$. But 
 $$|f_{n_k}(x_{n_k})-f(x_{n_k})|\le|f_{n_k}(x_{n_k})-f(x)|+|f(x)-f(x_{n_k})|$$ goes to zero by assumption and the continuity of $f$. This contractions proves the theorem.
\end{proof}

\begin{corollary}For $t>0$ and a continuous function $f$ on $G$, $P_tf$ is also continuous and 
 $$sup_{y\in G_n}|P_tf(y)-P_t^nf(y)|\to0.$$
\end{corollary}

\begin{theorem} \label{smp}
 $P_t$ is a Feller semigroup. In particular the process $B(t)$ satisfies the strong Markov property.
\end{theorem}
\begin{proof}
 By the previous corollary we know that  $P_t$ maps $C(G)$ into $C(G)$. First we show that it the family $\{P_t\}$ is a semigroup.
 
 From the Markov property of $B_n$ we have that $P^n_{t+s}=P^n_tP^n_s$. Therefore it is enough to show that $P^n_tP^n_sf(x_n)$ converge to $P_tP_sf(x)$ whenever $x_n\to x$.
 \begin{align*}
  |P^n_tP^n_sf(x_n)-P_tP_sf(x)|\\
  \le |P^n_tP^n_sf(x_n)-P^n_tP_sf(x_n)|+|P^n_tP_sf(x_n)-P_tP_sf(x_n)|\\+|P_tP_sf(x_n)-P_tP_sf(x)|
 \end{align*}
 Since the first term is bounded by  $\sup_{y\in G_n}|P^n_sf(y)-P_sf(y)|$ it goes to 0 by the previous corollary. Similarly, the second term converge to zero since $P_sf$ is continuous. The last term vanishes since $P_tP_sf$ is continuous. 

 Since,  $B(t)$ is continuous and $B(0)=x$, we have $P_tf(x)\to f(x)$ for any continuous function $f$.
\end{proof}

\section{Cover Time} \label{cover}

The (expected) \defi{cover time} $CT_o(G)$ of a finite metric graph \g from a point $o\in G$ is the expected time untill standard \BM\ from $o$ on \g has visited every point of \G. The cover time of \g is $CT(G):= \sup_{o\in G} CT_o(G)$. It is proved in \cite{edgecov} that \ti\ an upper bound on $CT(G)$ depending only on the total length $\ell(G)$ of \g and not on its structure

\begin{theorem}[\cite{edgecov}] \label{GeWi}
For every finite graph $G$ and $\ell: E(G) \to \R_{>0}$, we have $CT(G)\leq 2\ell(G)^2$.
\end{theorem}

In this section we use this fact to deduce the corresponding statement for our \BM\ $B$ on a \gl\ continuum $X$: defining $CT(X)$ as above, with standard \BM\ replaced by our process $B$, we prove
\begin{theorem} \label{CT}
For every \gl\ continuum $X$ with $\Hm(X)=L<\infty$, we have $CT(X)\leq 20 L^2$.
\end{theorem}

In order to  prove it  we will need the following bound on the second moment of the cover time in terms of its expectation.

\begin{lemma} 
 \label{lem:second.moment}
Let \g be a finite metric graph, and denote by $\cov_x$ the (random) cover time from $x\in G$. Suppose that for a constant $Q \in \R$ we have $\Exp{\cov_x}\le Q$ for every $x\in G$. Then $\Exp{\cov_x^2}\le 24Q^2$ for every $x\in G$.
\end{lemma}
\begin{proof}
By the Chebyshev inequality we have $$\p{}{\cov_x\ge s}\le\Exp{\cov_x}/s\le Q/s,$$
for every $s$; setting $s=2Q$, we obtain 
\labtequ{p12}{$\p{}{\cov_x\ge 2Q}\le1/2.$}
We claim that \fe\ $k\in \N$ we have 
\labtequ{p2Qk}{$\p{}{\cov_x\ge 2Qk}\le(1/2)^k$.}
To see this, we subdivide time into intervals of length $2Q$. Since \eqref{p12} holds for every starting point $x$, the probability that in the $i$th time interval\\ $[(i-1)2Q, i2Q]$ the process fails to cover the whole space $G$ is at most $1/2$. Thus, if we run the process up to time $2Qk$, in which case we have $k$ such `trials', the probability of not covering $G$ in any of them is at most $(1/2)^k$, proving our claim. Note that we have been generous here, as we are ignoring the part of $G$ that was covered before the $i$th interval begins.

Using this, we can bound the second moment of $\cov$ as follows
 \begin{align*}
  \Exp{\cov_x^2}&=\int_0^{\8}2t\p{}{\cov_x\ge t}dt,
   \end{align*}
by Fubini's theorem. Splitting time $t$ into intervals of length $2Q$, the last integral can be rewritten as
   \begin{align*}
\sum_{k=0}^{\8} \int_{k2Q}^{(k+1)2Q}2t\p{}{\cov_x\ge t}dt
&\le 2\sum_{k=0}^{\8}\int_{k2Q}^{(k+1)2Q}t\p{}{\cov_x\ge k2Q}dt \\
= 2\sum_{k=0}^{\8}(2Q)^2 (k+1/2)\p{}{\cov_x\ge 2kQ}
  &\le 8Q^2\sum_{k=0}^{\8} (k+1/2)(1/2)^k = 24 Q^2.
 \end{align*}

\end{proof}

Using our bound for the second moment of $\tau$ from Lemma \ref{lem:second.moment} we can now bound the first moment:

\begin{lemma}
 \label{lem:covere}
Let  \seq{G} be a \gseq\ of a \gl\ continuum $X$. Suppose that for a constant $Q \in \R$ we have $\e{B_n}{\cov_x}\le Q$ for every $x\in G_n$. Then for every $x\in X$ $$\e{B}{\cov_x} \le 
10 Q.$$
\end{lemma}
\note{Konrad: please adapt the notation in \Lr{lem:covere} to that used above}
\begin{proof}
We would like to use the weak convergence of the law $\mu_{n}$ of \BM\ $B_n$ on $G_n$ to the law $\mu$ of our limit process $B$ (\Tr{thmain}) to deduce that $\e{x}{\cov}$ is finite from \Tr{GeWi}. However, we cannot do so directly as the cover time $\cov$ is not a continuous function from $C$ to $\R$. To overcome this difficulty, we introduce a function $h(t,\oo): C \to \R$ (parametrised by time $t$)  that is continuous and is closely related to $\cov$.

Let $r>0$ be some (small) real number. For a path $\omega\in C$, denote by $h'_r(t)[\omega]$ the 
total length of the set $\{x\in G \mid d(x,\oo(s))>r \text{ \fe\ } s\leq t\}$; in other words, if we thing of $\oo$ as the trajectory of a particle of `width' $r$, then $h'_r(t)[\omega]$ is the length of the part of \g that this particle has not covered by time $t$. We also define the normalised version $h_r(t)[\omega] := h'_r(t)[\omega]/L$, where $L$ is again the total length of \G. It is no loss of generality to assume that $L=1$. 

For every fixed $T,M\in \R$, the function $$\omega\mapsto\left(\int_0^T(h_r(t)[\omega])^{1/M}dt\right)^2$$ as a mapping from $C$ to $\R$ is continuous. We can now use the weak convergence of $\mu_{n,o}$ to $\mu_{o}$ to obtain
 \begin{align*}
 \e{x}{\left(\int_0^T(h_r(t))^{1/M}dt\right)^2}&=\lim_{n\to\8}\en{x}{\left(\int_0^T(h_r(t))^{1/M}dt\right)^2}\\
 &\le\lim_{n\to\8}\en{x}{\left(\int_0^T\1{h_r(t)>0}dt\right)^2},
 \end{align*}
 where we used the fact that $h_r(t) \le 1$.
Since $\ell(G) - \ell(G_n)$ converges to 0, we deduce that if a path \oo\ covers $G_n$ at time $t$, for sufficiently large $n$ compared to $r$, then  $h_r(t)[\oo]=0$. It follows that the expression in parenthesis can be bounded from above by $\cov$, and so by \Lr{lem:second.moment} we conclude that 
 \begin{align}
  \label{eq:noname1}
  \e{x}{\left(\int_0^T(h_r(t))^{1/M}dt\right)^2}\le \e{x}{\cov^2} \le 24Q^2. 
 \end{align}
 Now let $\eps>0$. Note that if $h_r(T)>\eps$, then $h_r(t)>\eps$ holds \fe\ $t< T$ since $h_r(t)$ is decreasing in $t$. This easily implies
 $$\e{x}{T^2\eps^{2/M}\1{h_r(T)>\eps}}\le\e{x}{\left(\int_0^T(h_r(t))^{1/M}dt\right)^2},$$
which combined with \eqref{eq:noname1} yields
 $$T^2\eps^{2/M}\p{x}{{h_r(T)>\eps}}\le 24Q^2.$$
As $M$ can be chosen arbitrarily large independently of $\eps$, we have
 $$\p{x}{{h_r(T)>\eps}}\le 24Q^2/T^2.$$
Letting  $\eps$ tend to 0 we deduce 
 $$\p{x}{{h_r(T)>0}}\le 24Q^2/T^2.$$ 
 
Observe that the events $\{h_r(T)>0\}$ decrease to $\{h_0(T)>0\}=\{\omega:\cov(\omega)>T\}$ as $r$ goes to 0. Hence $$\p{x}{\cov>T}\le 24Q^2/T^2.$$
 Finally, we have
\labtequ{quadbo}{$\e{x}{\cov}=\int_0^{\8}\p{x}{\cov>t}dt\le Q\sqrt{24}+\int_{Q\sqrt{24}}^{\8}24Q^2/t^2dt= 2\sqrt{24}Q<10Q$.}
 
%
\end{proof}

To prove \Tr{CT}, let \seq{G} be any \gseq\ of $X$. Note that $\ell(G_n)\leq \Hm(X)=:L$ \fe\ $n$ by the definition of $\Hm$. Thus we can plug the constant $Q= 2L^2$ from \Lr{GeWi} into \Lr{lem:covere} to obtain the bound $10Q = 20L^2$ on the cover time of $X$.

\begin{corollary}
 $B_t$ is positive recurrent.
\end{corollary}

\section{Further properties} \label{further} 

In this section we show that \HM\ on $X$ is stationary for our process, and that our process behaves locally like standard \BM\ on \R\ inside any edge of $X$.

Recall, that any edge $e \subset X$ can be viewed as an interval contained in the real line, that is, there is $F:e\to\R$ which is an isometry onto its image. The next lemma shows that our process $B$ locally coincides with the standard Brownian motion $W$.

\begin{proposition} \label{phi}
Let $e$ be an edge in $X$. For any continuous function $\phi$ with $k-1$ arguments each taking values in $e$, any increasing sequence $t_1,t_2,\dots,t_k$, and any $x\in e$, we have $$\Ex_x[\phi(F(B(t_1)),\dots,F(B(t_{k-1}))1_{t_k<\tau_{\partial e}}]=\Ex_{F(x)}[\phi(W(t_1),\dots,W(t_{k-1}))1_{t_k<\tau_{\partial F(e)}}]$$
\end{proposition}
\begin{proof}
Since the equation is true for $B_n$, we would like to pass to limit with $n$ to prove that $B$ also satisfies this, but first we have to deal  with the discontinuity of the indicator under the expectation sign. For any $\delta>0$ and $n$ we have
\begin{align*}
 \Ex_x[\phi(F(B^n(t_1)),\dots,F(B^n(t_{k-1}))dist(B^n[0,t_k],\partial e)^{\delta}]=\\
 \Ex_{F(x)}[\phi(W(t_1),\dots,W(t_{k-1}))dist(W[0,t_k],\partial F(e))^{\delta}]
\end{align*}
Since the function under the expectation sign is continuous, now we can pass to a limit with $n$ and next, by Lebesgue theorem, with $\delta$ to 0 proving the desired equality.
\end{proof}

\begin{proposition} \label{stat}
 The Hausdorff measure $\Hm$ on $X$ is the unique (up to multiplicative constant) invariant measure for process $B$.
\end{proposition}
\begin{proof}
Let \seq{G} be a \gseq\ of $X$. Then $\Hm_n:=\Hm(G_n)$ is a sum of lengths of edges of $G_n$, and it is proved in \cite{gl} that $\Hm(X)=\lim_n\Hm(G_n)$. Moreover, it is not hard to check that the measure $\Hm_n$ is invariant for $P^n_t$. Hence, by Lebesgue theorem, for any bounded continuous $f$, we have
 \begin{align*}
  \langle P_tf,\Hm \rangle=\lim_n\langle P^n_tf,\Hm \rangle=\lim_n\langle 1_{G_n}P^n_tf,\Hm \rangle\\=\lim_n\langle P^n_tf,\Hm_n \rangle=\lim_n\langle f,\Hm_n \rangle=\lim_n\langle 1_{G_n}f,\Hm \rangle=\langle f,\Hm \rangle.
 \end{align*}
 Since by Theorem \ref{CT} the process is recurrent $\Hm$ is the unique invariant measure (cf. \cite{khas}).
\end{proof}

\section{Outlook} \label{outl}

In this paper we constructed a diffusion $B$ on \gl\ spaces of finite length. The finite length condition plays an important role for the uniqueness of $B$, and it is indeed not hard to find \gl\ spaces of infinite length where the limit of the $B_n$ as in our construction depends on the choice of the \gseq\ \seq{G}.

An approach that can be used to try to avoid this situation, and hence extend our construction to spaces $X$ of infinite length, is to endow $X$ with a probability measure  $\mu$, and use this $\mu$ in order to control the speed of the $B_n$ as follows. Given any measured metric space $(X,d,\mu)$, and a diffusion $B: \R_+ \to Y$ on $Y$, one can consider the function $A_t:= \int_Y L_t(x) d\mu(x)$, where $L_t(x)$ denotes the local time of $B$ at $x$, and then reparametrize the diffusion by letting $B'(t)= B(A_t^{-1})$. This approach is standard in the study of diffusions on fractals; see e.g.\ \cite[Chapter 4]{BarDif}. (We would like to thank D.~Croydon for suggesting this approach.)

A further interesting quest would be to relate our process with the theory of Dirichlet forms of \cite{Fukushima}.

\acknowledgement{We would like to thank D.~Croydon for suggesting the approach described in \Sr{outl}. We are grateful to Wolfgang Woess and the Graz Institute of Technology for their hospitality which made this work possible.}

\bibliographystyle{plain}
\bibliography{collective}

\end{document}